\documentclass[12pt,reqno]{amsart}
\usepackage{graphicx}
\usepackage{amsmath}
\usepackage{amssymb}
\usepackage{amscd}
\usepackage{mathrsfs}
\usepackage[all,cmtip]{xy}
\usepackage{color}
\usepackage{pgfbaseimage}

\newtheorem{theorem}{Theorem}

\newtheorem{theoremc}{Theorem}

\newtheorem{rk}[theoremc]{Remark\!\!}
\newtheorem{cor}[theoremc]{Corollary\!\!}
\newtheorem{lem}[theorem]{Lemma}

\newtheorem{prop}[theorem]{Proposition}
\newcommand\bib[1]{\bibitem[#1]{#1}}

\renewcommand\1{{\bf 1}}

\newcommand\com[1]{}
\newcommand\C{{\mathbb C}}
\newcommand\Cc{{\let\mathcal\mathscr\mathcal C}}

\renewcommand\d{\delta}

\newcommand\DD{{\mathbb{D}}}

\newcommand\E{\mathcal{E}}

\newcommand\g{{\frak g}}

\newcommand\La{\Lambda}

\newcommand\N{{\mathbb N}}

\newcommand\oo{\omega}
\newcommand\op[1]{\mathop{\rm #1}\nolimits}
\newcommand\ot{\otimes}
\newcommand\p{\partial}

\newcommand\R{{\mathbb R}}
\renewcommand\t{\tau}

\newcommand\vp{\varphi}

\newcommand\z{\sigma}
\newcommand\Z{{\mathbb Z}}

\begin{document}

 \title[Poincar\'e function of geometric structures]{Poincar\'e function for moduli of
differential-geometric structures}
 \author{Boris Kruglikov}
 \date{}
\address{Department of Mathematics and Statistics, UiT the Arctic University of Norway, Troms\o\ 90-37, Norway.
\quad E-mail: {\tt boris.kruglikov@uit.no}. }
 \keywords{Differential Invariants, Invariant Derivations, conformal metric structure,
Hilbert polynomial, Poincar\'e function}

 \vspace{-14.5pt}
 \begin{abstract}
The Poincar\'e function is a compact form of counting moduli in local geometric problems.
We discuss its property in relation to V.\,Arnold's conjecture, and derive this conjecture 
in the case when the pseudogroup acts algebraically and transitively on the base.
Then we survey the known counting results for differential invariants and derive new formulae
for several other classification problems in geometry and analysis.
 \end{abstract}

 \maketitle

%%%%%%%%%
\section*{Introduction}

The Poincar\'e function counts the number of moduli in algebro-geometric problems.
Classically, for a graded algebra $A=\oplus_{i\ge0}A_i$ with $a_i=\dim A_i<\infty$,
this function is given by
 $$
P(z)=\sum_{i=0}^\infty a_i z^i.
 $$
In particular, this applies to the algebra of invariants $A=R[X]^G$
of an algebraic action of a Lie group $G$ on a variety $X$.
It encodes grows of the number of invariants with their algebraic degree.
% Similarly, the Poincar\'e function is defined for a graded $A$-module $M$.
For semi-simple Lie groups $G$ this has received numerous applications,
see e.g. description of % regular polyhedrons,
invariants and covariants of binary forms in \cite{Sp}.

In the same vein, the Poincar\'e function is used in the local analysis of differential-geometric problems.
As the setup let $G$ be an algebraic pseudogroup (the definitions will be recalled in Section \ref{S1})
acting on a space $\E$ of geometric objects, which can be the space of sections of a tensor bundle of
prescribed type or the sheaf of solutions to a certain geometric equation. However instead of considering
global sections or germs of those, we shall conveniently work with their jets.

Thus $\E$ consists of (jets of) sections of a bundle $\pi$ and $G$ consists of local diffeomorphisms of $J^0\pi$.
Prolong the action of $G$ to higher jets $J^k\pi$, possibly restricting to invariant subbundles $\E^k$ such that the
projections $\pi_{k,k-1}:\E^k\to\E^{k-1}$ are submersions. Thus $\E$ is a co-filtered manifold, also known as diffiety, the projective limit of $\E^k$. In most cases of interest $\E$ is either un-constrained or a formally integrable
differential equation. A more general setup will be given in Section \ref{S1}.

Due to algebraic nature of the action, it possesses a rational quotient $\mathcal{Q}_k=\E^k/G$
(space of $G$-orbits in $\E^k$) by the Rosenlicht theorem \cite{R} and the prolongation result of \cite{KL}.
Rational functions on $\mathcal{Q}_k$ are bijective with $G$-invariant rational functions on $\E^k$
and are called (global) scalar differential invariants of order $k$; their pole divisors are $G$-invariant.

Dimension $s_k$ of $\mathcal{Q}_k$ is equal to the transcendence degree of the field
of rational differential invariants of order $\le k$, and it corresponds to the number of
(functionally) independent invariants of such order. The difference $h_k=s_k-s_{k-1}$
can be interpreted as the number of ``pure order" $k$ differential invariants.

Under certain assumptions, the sequence $h_k$ (as well as $s_k$) or its arithmetic sub-sequence $h_{ak+b}$
(where $a\in\N$ is fixed, and $0\leq b<a$ varies) is a polynomial for $k\gg1$, called the Hilbert polynomial.
This fact was experimentally observed by V.\,Arnold for some local problems in analysis and geometry \cite{A}.
Later this conjecture was proved for the diffeomorphism pseudogroups acting on natural geometric bundles
in \cite{Sar}, and then for general algebraic pseudogroups acting on algebraic differential equations
in \cite{KL}. The basic assumption in the last reference, in addition to algebraicity, is transitivity of the
pseudogroup action on the base manifold, and we adapt this also in what follows.

The Poincar\'e function of this action is defined by
 \begin{equation}\label{P}
P(z)=\sum_{k=0}^\infty h_kz^k.
 \end{equation}
The series clearly converges and gives an analytic function in the disk $|z|<1$.
Under the above assumptions, it is a rational function with the only pole at $z=1$.
The Hilbert function is restored by the formula
 $$
h_k=\op{Res}\limits_{z=0}\frac{P(z)}{z^{k+1}}=\frac1{k!}\left.\frac{d^k}{dz^k}\right|_{z=0}P(z),
 $$
and the number of independent differential invariants of order $\leq k$ by the formula
(in both cases $k\gg1$)
 $$
s_k=\op{Res}\limits_{z=-1}\frac{P(z+1)}z-\op{Res}\limits_{z=0}\frac{P(z)}{z^{k+1}(z-1)}.
 $$
Thus Poincar\'e function is the generating function encoding the count for differential invariants.

Often, the quotient $\mathcal{Q}=\E/G$, co-filtered by $\mathcal{Q}_k$, has functional rank $\sigma$
and functional dimension $d$, meaning that the number of jets in $\mathcal{Q}$
is asymptotic to that for the space of jets of $\sigma$ functions of $d$ arguments:
$s_k\sim\sigma\cdot\binom{k+d-1}{d-1}$. This implies at once that
 \begin{equation}\label{PR}
P(z)=\frac{R(z)}{(1-z)^d},
 \end{equation}
for some polynomial $R(z)$, $R(1)\neq0$, so the functional dimension $d$ is easily identifyable.
The functional rank can be found by the change of variables $w=1-z$: $P(1-w)=\sigma w^{-d}+\sum_{i>-d} p_i w^i$
where the sum is finite. This yields $\sigma=R(1)$, so the functional dimension and rank express so:
 $$
d=-\lim_{z\to1}\frac{\log P(z)}{\log(1-z)},\qquad
\sigma=\lim_{z\to1}P(z)(1-z)^d.
 $$
More complicated Poincar\'e series than \eqref{PR}, leading to different formulas for $(d,\sigma)$,
are briefly discussed at the end of the paper, see Section \ref{S4}.

This paper has the following three goals:

% \noindent
 \begin{enumerate}
\item Discuss rationality of the function $P(z)$,
deducing a strong form of Arnold's conjecture \cite[Problem 1994-24]{A} in the case the action is algebraic
and transitive on the base, see Theorem \ref{T3};
\item Provide explicit rational formulae for $P(z)$ in many classical examples,
summarizing (sometimes correcting and generalizing) and compactifying the known results from the literature;
\item
Compute the Poincar\'e function for several new important cases, including an infinite type geometric structure,
which is a novel local result in almost complex geometry.
 % all previously known examples concerned finite type structures.
 \end{enumerate}
\noindent
These problems will be subsequently addressed in the further sections.
Validity of (1) relies on a derivation of the main result in \cite{KL}, which we recall in Section \ref{S1}
and then indicate modifications required to achieve the claim. The tools important to compute
the Poincar\'e function % $P(z)$
are given in Section \ref{S2}, and then the results of (2) and (3) are presented in
Section \ref{S3}, constituting the main body of this paper. An outlook is given in Section \ref{S4}.

%%%%%%%%
\section{On Arnold's conjecture}\label{S1}

A weak form of the Arnold conjecture states that the Poincar\'e function $P(z)$ is rational.
The coefficient $h_k=s_k-s_{k-1}$ of $z^k$ in \eqref{P} expresses through the codimension $s_k$ of generic
$G$-orbits in $\E^k$, but this $s_k$ can be also understood as codimension of a particular orbit
$G\cdot a_k$ through $a_k\in\E^k$ and then it depends on this point.
If the action is algebraic, then $s_k$ is constant on a Zariski open subset $\E'_k\subset\E^k$ \cite{R,KL}.
Uniting these yields a Zariski open set $\E'\subset\E$ on which $P(z)$ is rational \cite{KL}.

A strong form of Arnold's conjecture states that there exists a subset $\Sigma\subset\E$ of $\op{codim}\Sigma=\infty$
and a (co-filtered) stratification $\E\setminus\Sigma=\cup_\alpha\E_{\alpha}$ such that with
$s^{\alpha}_k=\op{codim}(G\cdot a_k\subset\E^k)$, $a_k\in\E^k_{\alpha}$, and
$h^{\alpha}_k=s^{\alpha}_k-s^{\alpha}_{k-1}$, the corresponding Poincar\'e function
$P_{\alpha}(z)=P(a_\infty;z)=\sum_{k=0}^\infty h^{\alpha}_kz^k$ is rational for every $\alpha$ (i.e.\ for every
$a_\infty=\{a_k\}_{k=0}^\infty\in\E^\infty_a$, $a=a_0\in M$).

%%%%%%%%
\subsection{A solution of the conjecture}

A pseudogroup is a collection of local diffeomorphisms $G\subset\op{Diff}_{\text{loc}}(M)$ that contains unit,
inverse, and composition whenever defined. It is called a Lie pseudogroup if its elements are solutions to a
system of differential equations, see \cite{Lie2,SS,Ku}. Thus we identify $G$ with a projective limit of subsets
$G^k\subset J^k(M,M)$ that give a formally integrable Lie equation.
Since local and formal diffeomorphisms have the same differential invariants (see below), we will
not make a distinction between them.

Denote by $J^k_n(M)$ the space of $k$-jets of $n$-dimensional submanifolds $N\subset M$; note that
$J^k(M,M)\subset J^k_m(M\times M)$ for $m=\dim M$. A differential equation $\E$ is a collection of
submanifolds $\E^k\subset J^k_n$, $\E^0=J^0_n=M$, such that the projections $\pi_{k,k-1}:\E^k\to\E^{k-1}$ are
submersions (note the un-constraint case: $\E^k=J^k_n$).
It is called formally integrable if $\E^k$ is a subset of the prolongation of $\E^{k-1}$, i.e.\
the defining relations of $\E^k$ are obtained by differentiations of those of $\E^{k-1}$.
Recall that there is a natural algebraic structure on fibers of $J^k_n$.
If the defining relations of $\E$ are algebraic (in jets of order $\ge1$; for simplicity, we assume no relation
of order zero is imposed on $\E$), the equation is called algebraic.

In particular, if the Lie equation is algebraic we call the pseudogroup $G$ algebraic. It naturally acts on the
jet-spaces $J^k_n$, and the equation $\E$ is called $G$-invariant if $G\cdot\E^k\subset\E^k$. Equivalently,
if $\mathcal{G}$ is the Lie algebra sheaf of $G$ (local vector fields $X=\frac{dg_t}{dt}|_{t=0}$ for paths
$g_t\subset G$, $g_0=\op{Id}$), then $\E$ is $G$-invariant if $X_{a_k}\in T_{a_k}\E^k$ for all $X\in\mathcal{G}$,
$a_k\in\E^k$.

A function $f$ on $\E$ is by definition a function $f:\E^k\to\R$ for some $k$ pulled back to $\E^\infty$.
It is $G$-invariant if $g^*f=f$ for any $g\in G$. Provided $G$ is connected,
this is equivalent to $L_Xf=0$ for any $X\in\mathcal{G}$.

Consider the field of rational functions $\mathfrak{R}(\E)=\cup_k\mathfrak{R}(\E^k)$
and its subfield of rational invariants $\mathfrak{F}=\mathfrak{R}(\E)^G$.
If $G$ is Zariski-connected (we assume this in what follows), $f$ is a rational differential invariant
iff $L_Xf=0$ for any $X\in\mathcal{G}$.
By Rosenlicht's theorem \cite{R}, elements of $\mathfrak{F}_k=\mathfrak{R}(\E^k)^G$ separate regular
$G$-orbits and the transcendence degree of $\mathfrak{F}_k$ is the codimension $s_k$ of a generic orbit in $\E^k$.

 \begin{rk}
By \cite{KL} there exists a natural number $l$ such that the subalgebra $\mathfrak{A}$ of invariant functions that are rational by jets of order $\leq l$ and polynomial by jets of higher order suffices to separate regular orbits.
 \end{rk}

In addition to differential invariants one defines invariant derivations as first order operators
in total derivatives $\nabla:\mathfrak{R}(\E)\to\mathfrak{R}(\E)$ commuting with the action of $G$.
Global Lie-Tresse theorem \cite{KL} states that the field $\mathfrak{F}$ (and the algebra $\mathfrak{A}$)
is generated by a finite number of differential invariants $I_i$ and a finite number of invariant derivations $\nabla_j$. Moreover, loc.cit.\ proves that the invariant syzygies and higher syzygies are also finitely generated in the Lie-Tresse
sense. This implies that $s_k$ and hence $h_k$ are polynomials in $k$ for $k\gg1$, whence the following claim \cite[Theorem 26]{KL}:

 \begin{theorem}\label{T1}
Consider an algebraic action of a connected pseudogroup $G$ on an irreducible algebraic differential equation $\E\subset J^\infty_n(M)$. Assume that $G$ acts transitively on $M$. Then the Poincar\'e function $P(z)$ of this action is rational and has form \eqref{PR}, where the degree $d$ of the only pole $z=1$ does not exceed the degree of the complex affine characteristic variety of $\E$; in particular $d\leq n$.
 \end{theorem}

This gives a solution of (the weak form of) Arnold's conjecture under the assumptions of the theorem, of which
the most crucial is the transitivity of $G$-action on the base $M$. We will comment at the end of the paper
on what happens when this assumption is violated.

%%%%%%%%
\subsection{A generalization: strong version of the conjecture}

We claim that the previous statement holds true for a more general class of submanifolds
$\E^\infty\subset J^\infty_n$ co-filtered by $\E^k\subset J^k_n$
as long as the basic assumptions of Theorem \ref{T1} are satisfied.

The setup is as follows. Let $\bar\E$ be a formally integrable differential equation co-filtered by
$\bar\E^k\subset J^k_n$. Consider a finite number of functions (nonlinear differential operators)
$\Phi_s:J^{k_s}_n\to\R$. Let $\sigma_k=\{s:k_s\le k\}$. Define
$\E^k=\{a_k\in\bar\E^k:\Phi_s(a_s)=0\,\forall s\in\sigma_k\}$. We assume
regularity: the projections $\pi_{k,k-1}:\E^k\to\E^{k-1}$ are submersions.

Thus we allow $\E^k$ to be not a part of the prolongation of $\E^{k-1}$, but this can happen only for a finite
set of orders $k$. We call such $\E$ a {\it generalized equation}. In particular, we can start
with $\bar\E^\infty=J^\infty_n$ and impose a finite number of differential equations $\{\Phi_s=0\}$ without including
prolongations of those. If $\Phi_s$ are algebraic functions and $\bar{\E}$ is algebraic, we call the generalized
equation $\E$ algebraic. $G$-invariance extends straightforwardly.

Consider a sequence of points $a_k\in\E^k$, $\pi_{k,k-1}(a_k)=a_{k-1}$, and let $a_\infty=\lim a_k\in\E^\infty$.
If $a_1=[N]^1_a$ for a $n$-manifold $N\subset M$ then we denote $\tau_a=T_aN$ and $\nu_a=T_aM/T_aN$; they depend
only on $a_1$.
As for usual differential equations $g_k(a_k)=\op{Ker}(d\pi_{k,k-1}:T_{a_k}\E^k\to T_{a_{k-1}}\E^{k-1})$
is called the $k$-symbol of $\E$, and it is naturally identified with a subspace in
$S^k\tau_a^*\otimes \nu_a=\op{Ker}(d\pi_{k,k-1}:T_{a_k}J_n^k\to T_{a_{k-1}}J_n^{k-1})$.
Uniting these we get the symbolic system $g(a_\infty)=\{g_k(a_k)\}\subset S\tau_a^*\otimes\nu_a$.

Let $\delta:S^i\tau_a^*\otimes\nu_a\ot\La^j\tau^*\to S^{i-1}\tau_a^*\otimes\nu_a\ot\La^{j+1}\tau^*$
be the Spencer $\delta$-differential (symbol of the de Rham operator).
When $\E$ is a differential equation and $g$ its symbol, the sequence
 \begin{equation}\label{Spe}
\cdots\to g_{i+1}\ot\La^{j-1}\t^*\stackrel\d\longrightarrow
g_i\ot\La^j\t^*\stackrel\d\longrightarrow
g_{i-1}\ot\La^{j+1}\t^*\stackrel\d\to\cdots
 \end{equation}
is the Spencer complex; its cohomology at the $(i,j)$-term $H^{i,j}(\E;a_\infty)=H^{i,j}(g)$ is called the Spencer $\d$-cohomology group.

For a generalized equation $\E$ the map $\delta$ on $g_i\ot\La^j\tau^*$ may not take values in
$g_{i-1}\ot\La^{j+1}\tau^*$ when $i$ is an order, i.e.\ $\sigma_i\neq\sigma_{i-1}$. However for $i$
exceeding the maximum order the $\d$-differential is well-defined, so if, in addition, $i$ exceeds
the involutivity order of $\bar{\E}$ then $H^{i,*}(g)=0$.

 \begin{theorem}\label{T2}
Consider an algebraic action of a pseudogroup $G$ on an algebraic generalized differential equation $\E\subset J^\infty_n(M)$. Let $G$ act transitively on $M$. Then the Poincar\'e function $P(z)$ of this action is rational of the form \eqref{PR}. It has only one pole at $z=1$ of degree $d\leq n$. Moreover, $P(z)=P(a_\infty;z)$
is locally constant by $a_\infty$ when this point vary in a component of a Zariski open set $\E''\subset\E$.
 \end{theorem}

 \begin{proof}
Note that we allow $\E$ to be reducible. In this case we restrict to one of its finitely many components.
Thus the claim follows from an irreducible case, on which we now concentrate.

Let $\Delta_k(a_k)=T_{a_k}(G^k\cdot a_k)=\{X^{(k)}_{a_k}:X\in\mathcal{G}\}$ be the tangent differential system.
For a point $a_\infty\in\E^\infty$ consider the subspace
 $$
\varpi_k=\op{Ker}\bigl(d\pi_{k,k-1}:\Delta_k(a_k)\to\Delta_{k-1}(a_{k-1})\bigr)\subset S^k\tau_a^*\otimes \nu_a.
 $$

The main observation is that the proofs of Proposition 10 and Theorem 11 of \cite{KL} use only
the surjectivity of $\pi_{k,k-1}:\E^k\to\E^{k-1}$ and algebraicity of the action.
Thus we we can apply Corollary 12 of loc.cit. to conclude that there exists a natural $l$ and a Zariski open subset
$\E''=\pi_{\infty,l}^{-1}(\E''_l)\subset\E$ such that for all $i\ge l,j\ge0$ and $a_\infty\in\E''$ the sequence
 \begin{equation}\label{Spe}
\cdots\to \varpi_{i+1}\ot\La^{j-1}\t^*\stackrel\d\longrightarrow
\varpi_i\ot\La^j\t^*\stackrel\d\longrightarrow
\varpi_{i-1}\ot\La^{j+1}\t^*\stackrel\d\to\cdots
 \end{equation}
is well-defined and is exact: $H^{i,j}(\varpi)=0$.

Denote $\mathfrak{d}_k=\op{Ker}(T\mathcal{Q}_k\to T\mathcal{Q}_{k-1})$, where $\mathcal{Q}_k=\E^k/G$ is the
rational quotient. Then the exact sequences
 $$
0\to \varpi_k\longrightarrow g_k \longrightarrow \mathfrak{d}_k\to 0
 $$
and the corresponding Spencer $\d$-complexes unite into a bi-complex, which by the snake lemma implies that
$H^{i,j+1}(\varpi)=H^{i+1,j}(\mathfrak{d})$ for large $i$, in the range where $H^{i,j+1}(g)=H^{i+1,j}(g)=0$
(we can assume $H^{i,*}(g)=0$ for $i\ge l$).
Hence $H^{i,j}(\mathfrak{d})=0$ for $i\gg0$, cf.\ \cite[Theorem 16]{KL}.
Thus $\dim\mathfrak{d}_k$ grows polynomially for $k\gg0$ and this implies that $h_k$  grows polynomially in the same
range, whence the claim.
 \end{proof}

Let us note that we have not used Lie-Tresse theorem for the generalized equation $\E$ in this proof, but it
generalizes to this case as well.

Now we derive a version of Arnold's strong conjecture.

 \begin{theorem}\label{T3}
Let an algebraic pseudogroup $G$ act transitively on a manifold $M$ and its prolonged action preserve
an algebraic differential equation $\E\subset J^\infty_n(M)$.
Then there exist a subset $\Sigma\subset\E$ of $\op{codim}\Sigma=\infty$
and an algebraic stratification $\E\setminus\Sigma=\cup_\alpha\E_\alpha$ such that for every $\alpha$
the Poincar\'e function $P_\alpha(z)=P(a_\infty;z)$ is rational with the only pole at $z=1$ of degree $d\leq n$.
This $P_\alpha$ depends only on $\alpha$ and not on $a_\infty\in\E_\alpha$.
 \end{theorem}

Note that a differential equation (taken together with all prolongations)
$\E$ is itself of infinite codimension in $J^\infty_n(M)$ unless $\E=J^\infty_n(M)$,
but codimension of $\Sigma$ is measured in $\E$.

 \begin{proof}
Let us begin with $\E$. By Theorem \ref{T2} the Poincar\'e function $P(z)=P(a_\infty;z)$
is of the required type as long as $a_\infty$ belongs to a Zariski open set $\E''$. The complement
$\tilde\E=\E\setminus\E''$ is a Zariski closed subset of $\E$. If it is of infinite codimension, we are done.
Otherwise it is a stratified algebraic generalized equation invariant under the action of $G$, and the
assumptions of Theorem \ref{T2} are satisfied (in particular, the action is transitive on the base).
It can happen that in addition to equalities, specifying $\E\subset\bar{\E}$ in the
preceding proof, we introduce inequalities, but the conclusion will not suffer from this.
Thus we apply Theorem \ref{T2} again and obtain rationality of the Poincar\'e function on a Zariski open
subset $\tilde\E''$ of $\tilde\E$. Continuing in the same way for at most countably many steps, we conclude the claim.
 \end{proof}

Let us give an example of a situation, where assumptions of the previous theorem fail and the conclusion is different.
Consider vector fields on a manifold $N$ as sections of its tangent bundle.
Here $M=TN$, $n=\dim N$, and we consider only jets of sections of $\pi:E=M\to N$,
restricting to $J^\infty(N,TN)\subset J^\infty_n(M)$. The group $G=\op{Diff}_{\text{loc}}(N)$ naturally (and
algebraically) acts on this jet-space. However $G$ does not act transitively on $M$: there is an open
orbit $U=TN\setminus N$ and the zero section $0_N\equiv N$. In the preimage
$\pi_{\infty,0}^{-1}(U)$ the Poincar\'e function is rational. In fact, it equals $P(z)=0$.

However for the points $a_\infty$ with $a=\pi_{\infty,0}(a_\infty)\in 0_N$ the normal form theory applies, and
$P(a_\infty;z)$ depends essentially on the jet $a_\infty$. In this case, Arnold's conjecture is plausible, but
the Poincar\'e function varies with $a_\infty$: in non-resonant case (depends on $a_1$ only) the vector field
is formally linearizable and so $P(z)=nz$, while the resonant formal normal form can lead to
poles at other points on the unit circle $|z|=1$. We will discuss this phenomenon closer in the Conclusion.

%%%%%%%%
\subsection{An example of computation}
Consider the action
 $$
g:(x,y,u)\mapsto(X(x,y),y+c_1,u+c_2)
 $$
of the pseudogroup $G=\{g\}$ on $M=\R^3(x,y,u)=\R^2(x,y)\times\R^1(u)$ and prolong it
to $J^\infty(\R^2,\R)=\R^\infty(x,y,\{u_{i,j}\}_{i,j\ge0})$,
where $u_{i,j}$ is the jet-coordinate corresponding to $D_x^iD_y^ju(x,y)$.
The Lie algebra sheaf of $G$ is $\mathcal{G}=\langle f(x,y)\p_x,\p_y,\p_u\rangle$.
 % \simeq\R\rtimes C^\infty(\R^2)\times\R
Note that the action is algebraic and transitive on the base, so all assumptions are satisfied.

The isotropy subalgebra in $\mathcal{G}$ of the point $a=0$ in $M$ is $\mathcal{G}_a=\{X=f(x,y)\p_x:f(0,0)=0\}$.
Note that prolongation of such $X$ to $J^\infty$ is
 \begin{equation}\label{Xprol}
X^{(\infty)}_a=-\sum_{i+j>0}D_x^iD_y^j(f(x,y)u_{10})\p_{u_{i,j}}.
 \end{equation}
This action has a unique open orbit -- the complement of the stratum $\Sigma_1=\{u_{10}=0\}$, i.e.\ $P(z)=0$ on
$J^\infty\setminus\Sigma_1$. Indeed, the prolonged field to $k$-jets is
 $$
X^{(k)}_a=-\sum_{0<i+j\le k}(f_{i,j}u_{10}+\dots)\p_{u_{i,j}},
 $$
where dots denote the lower jets of the group parameter $f$. Varying these jets makes
the coefficients of $\p_{u_{i,j}}$ arbitrary provided $u_{10}\neq0$. Thus the orbit in $J^k$ is open,
and all differential invariants are constants.
 %(uniqueness of the open orbit follows because every function $u=u(x,y)$ with $u_x\neq0$ is locally $G$-equivalent to $u=x$).

Consider the singular stratum $\Sigma_1=\{u_{10}=0\}$ (codimension 1, no prolongations).
In this case
 \begin{equation}\label{2j}
X^{(k)}_a=-\sum_{1<i+j\le k}(if_{i-1,j}u_{20}+jf_{i,j-1}u_{11}+\dots)\p_{u_{i,j}},
 \end{equation}
where dots denote the lower jets of $f$. Counting the group parameters we see that if $u_{20}\neq0$
there is one pure order differential invariant in every order: $h_k=1$ for $k>0$. The first invariants are:
 $$
I_1=u_{01},\ I_2=u_{02}-\frac{u_{11}^2}{u_{20}},\
I_3=u_{03}-\frac{u_{11}^3}{u_{20}^3}u_{30}+3\frac{u_{11}^2}{u_{20}^2}u_{21}-3\frac{u_{11}}{u_{20}}u_{12}.
 $$

The next singular stratum is $\Sigma_2=\{u_{10}=0,u_{20}=0\}$. In this case a similar argument implies
that provided $u_{11}\neq0$ there is one pure order differential invariant $h_k=1$ in every order $0<k\neq2$,
and for $k=2$ we have $h_2=0$.
The first invariants are:
 $$
I_1=u_{01},\ I_3=\frac{u_{30}}{u_{11}^3},\
I_4=\frac{u_{40}}{u_{11}^4}-6\frac{u_{30}u_{21}}{u_{11}^5}+3\frac{u_{02}u_{30}^2}{u_{11}^6}.
 $$

The next singular stratum is $\Sigma_3=\{u_{10}=0,u_{20}=0,u_{11}=0\}$. Here the same argument implies
that provided $u_{30}\neq0$ we have $h_1=h_2=1$ and $h_k=2$ for $k>2$, so that we obtain two new invariants
in every order starting from order three. The first invariants are:
 $$
I_1=u_{01},\ I_2=u_{02},\
I_{3a}=u_{03}+2\frac{u_{21}^3}{u_{30}^2}-3\frac{u_{21}u_{12}}{u_{30}},\
I_{3b}=\frac{(u_{30}u_{12}-u_{21}^2)^3}{u_{30}^4}.
 $$

In the same way we obtain all further singular strata $\Sigma_4=\{u_{10}=0,u_{20}=0,u_{11}=0,u_{30}=0\}$,
$\Sigma_5=\{u_{10}=0,u_{20}=0,u_{11}=0,u_{30}=0,u_{21}=0\}$,
$\Sigma_6=\{u_{10}=0,u_{20}=0,u_{11}=0,u_{30}=0,u_{21}=0,u_{12}=0\}$, etc.
In the limit we get the stratum $\Sigma_\infty=\{u_{1+i,j}=0:i,j\ge0\}$, which is the
infinitely prolonged equation $\{u_x=0\}$. In this latter stratum the group reduces to three translations on the base,
so all jet-coordinates $u_{01},u_{02},u_{03},\dots$ are differential invariants. We summarize our computations in the following table:

\begin{center}
 \begin{tabular}{|l|l|}
\hline
$\Sigma_0\setminus\Sigma_1$ & $P(z)=0$ \\
$\Sigma_1\setminus\Sigma_2$ & $P(z)=z+z^2+z^3+\dots=\frac{z}{1-z}$ \\
$\Sigma_2\setminus\Sigma_3$ & $P(z)=z+z^3+z^4+z^5+\dots=\frac{z-z^2+z^3}{1-z}$ \\
$\Sigma_3\setminus\Sigma_4$ & $P(z)=z+z^2+2z^3+2z^4+2z^5+\dots=\frac{z+z^3}{1-z}$ \\
$\Sigma_4\setminus\Sigma_5$ & $P(z)=z+z^2+z^3+2z^4+2z^5+2z^6+\dots=\frac{z+z^4}{1-z}$ \\
$\Sigma_5\setminus\Sigma_6$ & $P(z)=z+z^2+z^4+z^5+z^6\dots=\frac{z-z^3+z^4}{1-z}$ \\
$\Sigma_6\setminus\Sigma_7$ & $P(z)=z+z^2+z^3+3z^4+3z^5+3z^6+\dots=\frac{z+2z^4}{1-z}$ \\
\quad\dots & \qquad \dots \qquad \dots \qquad \dots \\
\quad$\Sigma_\infty$ & $P(z)=z+z^2+z^3+z^4+z^5+\dots=\frac{z}{1-z}$ \\
\hline
 \end{tabular}
\end{center}

The orbit foliation can have complicated singularities. Let us demonstrate this on example of the stratum
$\Sigma_1$. For a point $a$ in it consider $\R^3=\pi_{2,1}^{-1}(a_1)$, with coordinates
$r=u_{20}$, $s=u_{11}$, $t=u_{02}$. From formula \eqref{2j} (in this case there will be no dots and summation
is by $i+j=2$), the Lie algebra sheaf $\mathcal{G}$ on it is given by two vector fields
$X=2r\p_r+s\p_s$, $Y=r\p_s+2s\p_t$ (coefficients of $f_{10}$ and $f_{01}$ respectively).

The distribution $\langle X,Y\rangle$ is involutive and its foliation is shown below.
The stratification is as follows: $r\neq0$ (invariant $I=t-\frac{s^2}t$), $r=0,s\neq0$ (no invariants),
$r=s=0$ (invariant $I=t$).

 \vspace{-0.2cm}
 \begin{center}
\includegraphics[height=5cm,width=9cm]{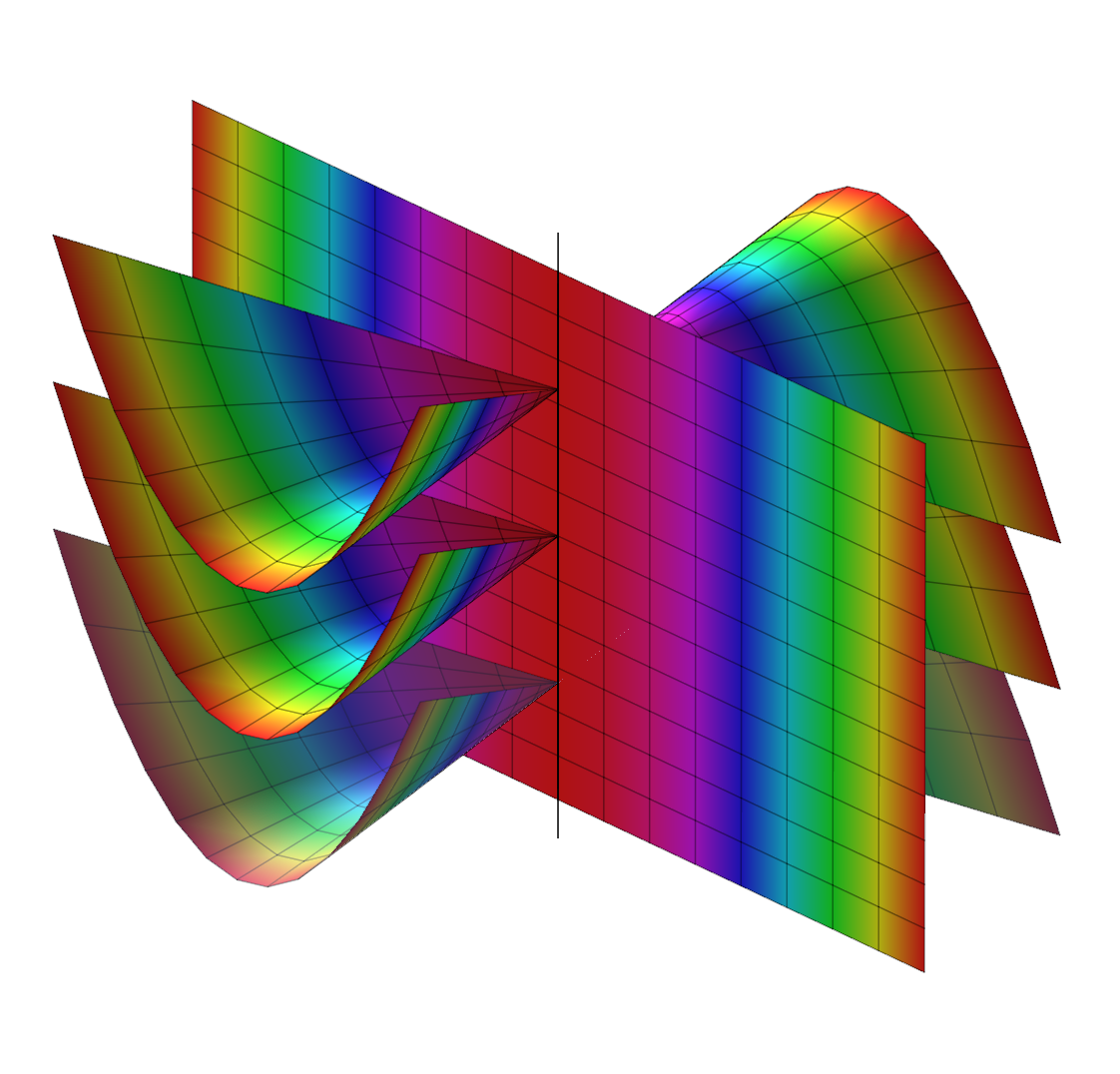}
 \end{center}
% \vspace{-1cm}

%%%%%%%%
\section{On computation of the Poincar\'e function}\label{S2}

In this and next sections we consider a natural bundle $\pi:E\to M$, and $E$
(or an open subset of it) will play the role of $M$ from the previous section.
From now on $M$ will be the base of the bundle $\pi$, and $G\subseteqq\op{Diff}_{\text{loc}}(M)$ a
pseudogroup on $M$. In what follows an equality will be our primary interest (the other cases will be
explicitly specified), so we specify the method to this case. Note that, by naturality of $\pi$, the action
of $G$ lifts from $M$ to $E$, and hence $G\subsetneqq\op{Diff}_{\text{loc}}(E)$.

Denote by $\DD^k_a=\{[\vp]_a^k:\vp\in G=\op{Diff}_{\text{loc}}(M),\vp(a)=a\}$ the so-called differential group
of order $k$ on $M$ at $a$. If the above lift involves $r$ differentiations, i.e.\ the $G$-action on $E$ has order $r$,
then $\DD_a^{k+r}$ acts on $J^k_a\pi$. In almost all our examples the action will be transitive on $E=J^0\pi$
(the opposite will be explicitly noted). Moreover the lift from $M$ to $E$ and further prolongations will keep
the action algebraic, so the assumptions of the previous section hold.

In some cases, we deal with pure jets, but in some others a differential equation is given, restricting the action
to $\E^k_a\subseteq J^k_a\pi$ (un-restricted case corresponds to the equality). Thus, abbreviating $T_a=T_aM$,
the fiber $\Delta_{k+r}\simeq S^{k+r}T^*_a\ot T_a$ of the projection $\DD_a^{k+r}\to\DD_a^{k+r-1}$ over
the unit (identity) acts on the symbol $g_k(a_k)=T_{a_k}\E^k\to T_{a_{k-1}}\E^{k-1}$;
here $a_k$ is a sequence of jets in $\E^k_a$ with projective limit $a_\infty$, i.e.\ $\pi_{k,k-1}a_k=a_{k-1}$.

Note that if $X$ is a vector field vanishing to order $k+r$ at $a$ and $s$ is a (local) section of $\pi$, then
$L_X(s)$ vanishes to order $k$ at $a$. Denoting $\lambda=[X]_a^{k+r}\in\Delta_{k+r}$ the corresponding jet
and $\rho:S^rT^*_a\ot T_a\to T_{a_0}E_a$ % (if $\pi$ is a vector bundle, this is $E_a$)
the symbol of the action, we have for a section with $s(a)=a_0$: $[L_Xs]^k_a=\zeta^k_{a_0}(\lambda)$,
where $\zeta^k_{a_0}$ is the composition of the canonical splitting map $\delta_k$ and the prolongation of the symbol map,
 $$
\zeta^k_{a_0}:S^{k+r}T^*_a\ot T_a \stackrel{\delta_k}\longrightarrow S^kT^*_a\ot S^rT^*_a\ot T_a
\stackrel{1\ot\rho}\longrightarrow S^kT^*_a\ot T_{a_0}E_a.
 $$
This implies the following. Denote $\g_r=\op{Ker}(\rho:S^rT^*_a\ot T_a\to T_{a_0}E_a)$ and let
$\g_{k+r}=\g_r^{(k)}=S^kT^*_a\ot\g_r\cap S^{k+r}T^*_a\ot T_a$ be its prolongation
 % note that $\op{Ker}(\rho_k^r)=\g_r^{(k)}$ if the first Spencer cohomology vanishes, $H^{1,j}(\g)=0$ for $j\ge r$
(to complete this symbolic system, we let $\g_i=S^iT^*_a\ot T_a$ for $0\leq i<r$).

 \begin{prop}\label{pLY}
The tangent space to the orbit $\Delta_{k+r}(a_k)\subset T_{a_k}\E^k$ is equal to the image $\op{Im}(\zeta^k_{a_0})$.
Moreover, $\op{Ker}(\zeta^k_{a_0})=\g_{k+r}$ and the normal space to the orbit is $\op{CoKer}(\zeta^k_{a_0})$. \qed
 \end{prop}

For $r=1$ the map $\delta=\delta_1$ is the usual Spencer differential and the above prolongation is
the standard Sternberg prolongation of first order structures.
Note that in presence of equation $\E$ the size of
kernel and cokernel of the map $\zeta^k_{a_0}$ may vary, we will comment in such cases.

In the case of Riemannian structures, when $r=1$ and $\g=\mathfrak{so}(n)$,
the above proposition was proved in \cite{LY}. In this case $\g_i=0$ for $i\ge2$.
In the case of symmetric connections (considered in details in the next section), $r=2$ and the symbol
$\rho:S^2T^*_a\ot T_a\to S^2T^*_a\ot T_a$ is an isomorphism, whence $\g_i=0$ for $i\ge2$.

Denote by $\op{St}^{k+r}_{a_k}\subset\DD^{k+r}_a$ the connected component of unity in the stabilizer of $a_k\in\E^k_a$.
 % Let us call an equation $\E$ regular, if generic $G_a$-orbits in $\E^k_a$ are regular orbits in $J^k_a$
 % (by transitivity of the action this does not depend on the choice of $a\in M$).
We have: $\op{Ker}(d\pi_{k+r,k+r-1}|\op{St}^{k+r}_{a_k})=\op{St}^{k+r}_{a_k}\cap\Delta^{k+r}=\op{Ker}(\zeta^k_{a_0})$.
Then Proposition \ref{pLY} implies:

 \begin{cor}
Assume that equation $\E$ is regular.
If $\g$ is of finite type and $\g_i=0$ for $i\ge l$, then $\op{Ker}(\zeta^k_{a_0})=0$ for $k\ge l-r$.
Consequently, the projection $\pi_{k+r,k+r-1}:\op{St}^{k+r}_{a_k}\to \op{St}^{k+r-1}_{a_{k-1}}$ is injective
for such $k$. \qed
 \end{cor}

The action is {\it locally free\/} from the jet-level $l$ if $\op{St}^{k+r}_{a_k}=0$ for $k\ge l$ and Zariski-generic
$a_k$. Note that in all cases of finite type we consider in the next section, the stabilizer will be resolved in
a finite number of prolongations. This has the following explanation. Since the Lie equation is of finite type
at generic jet of the geometric structure, stabilization of $\op{St}^{k+r}_{a_k}$ at non-zero space would
imply non-trivial local symmetry of the considered structure, while in all our examples generic geometric
structures will have only trivial local symmetries.

%%%%%%%%
\section{A panorama of examples}\label{S3}

Below we compute the Poincar\'e function of most popular geometric structures, whose moduli arise in applications.
Some of these formulae have been known before and we provide a reference, for some others only the orbit
dimensions have been known and we derive a compact formula for the Poincar\'e series
(usage of Maple is acknowledged at that stage). We also correct errors in several previous works on the subject,
and provide short computations based on prolongation technique of the previous section. Finally,
we add new examples: metric connections, Weyl conformal and almost complex structures. The latter case is
especially interesting as the first non-trivial structure of infinite type with novel effects in
local moduli count.

We denote $n=\dim M$ in all computations in this section, unless otherwise specified
(dimension $d$ if $M$ will be indicated as $M^d$).

 %%% %%%
\subsection{Second order ODE modulo point transformations}\label{S31}

This is one of the most known classical examples, where differential invariants have been computed and counted.
Our only contribution here are the formulae for the Poincar\'e function.

% % % % %
\subsubsection{{\bf General second order ODE}}\label{1a}

Consider differential equations $y''=f(x,y,y')$ given by a function $f$ of three variables. The action
of the pseudogroup $G=\op{Diff}_{\text{loc}}(\R^2)$ on the space of independent and dependent variables $(x,y)$
prolongs to the space $\R^4(x,y,y',y'')$. In this case $M=\R^2$ and $E=\R^4$.
Denoting $y'=p$, $y''=u$, we get a transitive algebraic action
on the space $J^0=\R^3(x,y,p)\times\R^1(u)$, which prolongs to the action on the space $J^\infty(\R^3)$ of jets of functions $u=f(x,y,p)$.

The problem of differential invariants of this action was initiated by S.\ Lie \cite{Lie1},
and all relative invariants were found by A.\ Tresse \cite{Tr2}. The absolute differential invariants
were derived and counted in \cite{K1}: $h_k=0$ for $k\leq4$, $h_5=3$ and $h_k=\binom{k}{2}-4$ for $k>5$.
Therefore we obtain
 $$
P(z)=3z^5+\sum_{k=5}^\infty\Bigl(\frac{k(k-1)}2-4\Bigr)z^k=\frac{z^5(3+2z-7z^2+3z^3)}{(1-z)^3}.
 $$
There are no differential invariants of order up to four: $G$ acts transitively on $J^3$, and has a
Zariski open orbit in $J^4$ -- its complement is a reducible algebraic variety $I\cdot H=0$, where
$I,H$ are basic relative invariants. The numbers $s_k=\sum_{i\leq k}h_i$ correspond to codimension of the
orbit in the domain $I\neq0,H\neq0$ of $k$-jets.

% % % % %
\subsubsection{{\bf Second order ODE cubic in $y'$}}\label{1b}

The singular stratum given by $H=0$ is dual by E.\ Cartan \cite{C} to $I=0$, so it is enough to consider only
the latter. This relative invariant has a simple formula $I=\frac{\p^4f}{\p p^4}$, so $I=0$ is equivalent to
cubic dependence of the right-hand side of the ODE on $p=y'$:
$y''=\alpha_0(x,y)+\alpha_1(x,y)y'+\alpha_2(x,y)(y')^2+\alpha_3(x,y)(y')^3$.
Such ODEs are equivalent to projective connections in 2D \cite{C}.

The group $G$ acts on $J^0=\R^2(x,y)\times\R^4(\alpha_0,\alpha_1,\alpha_2,\alpha_3)$
and the action prolongs to $J^\infty(\R^2,\R^4)$.
This action is transitive in 2-jets, and transitive outside the stratum $F_3=0$ in 3-jets, where $F_3$ is the Liouville
relative invariant \cite{Lio}, see also \cite{K1}. Differential invariants of this action were counted in
\cite{Tr1,Y}: $h_k=0$ for $k<4$, $h_k=2(k-1)$ for $k\ge4$. Therefore we obtain
 $$
P(z)=\sum_{k=4}^\infty2(k-1)z^k=\frac{2z^4(3-2z)}{(1-z)^2}.
 $$

% % % % %
\subsubsection{{\bf Second order ODEs of special Lie form}}\label{1c}

% The hierarchy of singular strata in the second order ODE problem
% has a lot of development, yet counting and classification of invariants in all cases has not been finished.
The following class of equations was introduced by S.\ Lie: $y''=f(x,y)$.
It includes all Painlev\'e transcendents (after a point transformation \cite{Bor})
and so is of special importance. % It has been thoroughly investigated.
The point transformation pseudogroup leaving the class
invariant is $(x,y)\mapsto(X(x),cX'(x)^{1/2}y+Y(x))$, it naturally extends to the space $J^0=\R^2(x,y)\times\R(f)$.
Differential invariants of this action were computed by P.\ Bibikov \cite{Bi}. In particular,
$h_k=0$ for $k\leq3$, $h_4=2$, and $h_k=k-1$ for $k\ge5$. This implies the formula
 $$
P(z)=2z^4+\sum_{k=5}^\infty(k-1)z^k=\frac{z^4(2-z^2)}{(1-z)^2}.
 $$

 %%% %%%
\subsection{Metric and related structures}\label{S32}

Consider $E=S^2T^*M$. The group $G$ acts in the fiber $S^2T^*_xM$ through the general linear group $GL(T_xM)$.
The action is not transitive, and degenerate quadrics form a singular stratum $\Sigma_x$. The complement
to $\Sigma=\cup_x\Sigma_x$ in $E$ is one orbit of $G=\op{Diff}_{\text{loc}}(M)$ over $\C$, while over $\R$
it splits into a finite union of orbits numerated by the index. Resetting $E$ to be one of those domains
we get an algebraic fiber bundle on which $G$ acts transitively. Sections of it correspond to (pseudo-)Riemannian
metrics $g$. Note that the number of differential invariants is independent of the index of $g$ and so the Poinar\'e
function is the same for pseudo-Riemannian metrics as for Riemannian ones.

Below we study a bundle of metrics or a differential equation in it; we also impose additional a
complex or tri-complex structure, constrained by the known relations. The bundle $E$ is properly modified.

% % % % %
\subsubsection{{\bf Riemannian metrics on $M^n$}}\label{2a}

Local scalar differential invariants of metrics for $n=2$ were studied by K. Zorawski \cite{Z}, and their count is:
$h_k=0$ for $k<2$, $h_2=h_3=1$, $h_k=k-1$ for $k>3$.

When $n>2$, the count of invariants was done by C.\,N.\ Haskins \cite{Ha}:
$h_k=0$ for $k<2$, $h_2=\frac12\binom{n}{3}(n+3)$ and
$h_k=\binom{n+1}{2}\binom{n+k-1}{k}-n\binom{n+k}{k+1}$ for $k>2$. This implies the formula \cite{K2}:
 $$
P(z)=\left\{
\begin{array}{ll}
\frac{z^2(1-z+2z^2-z^3)}{(1-z)^2}, & \text{ for }n=2,\vphantom{\frac{\frac22}{\frac22}}\\
\frac{n}z+\binom{n}2\cdot(1-z^2)-\frac1{(1-z)^n}\cdot\bigl(\frac{n}z-\binom{n+1}2\bigr), & \text{ for }n>2
\vphantom{\frac{\frac22}{\frac22}}.
\end{array}
\right.
 $$
Note that singularity at $z=0$ is inessential and is used here (and below) for compactification of the answer.

% % % % %
\subsubsection{{\bf Einstein metrics}}\label{2b}

This is an important special stratum. Note that the Einstein condition $\op{Ric}_g=\Lambda g$ for some
$\Lambda\in C^\infty(M)$ is an equation $\E$ on the sections of the bundle $E$ from the general case.
Recall that $\Lambda$ is constant and non-trivial cases arise for $n\ge4$ (indeed, $P(z)=z^2$ for $n=2,3$).
The description and count of differential invariants in the 4D case was done by V.\ Lychagin and V.\ Yumaguzhin
\cite{LY}. Their method extends further, as follows.

 \begin{prop}
We have: $h_k=0$ for $k<2$, $h_2=\frac1{12}(n^2-1)(n^2-12)$ and $h_k=
\frac{(k-1)n(n+k-1)(n+2k-2)}{2(k+1)(n-2)}\binom{n+k-4}{k}$ for $k>2$.
 \end{prop}

 \begin{proof}
In this case $r=1$, $a_0=g$ is a (pseudo-) Riemannian metric and similar to \cite{LY} (beware of
different indexing convention for stabilizers) we compute $\op{St}^2_{a_1}\simeq\op{St}^1_{a_0}=SO(g)$.
Since $a_2$ encodes the curvature tensor that, for Einstein metrics, consists of the scalar curvature and
the Weyl tensor, for generic $a_2$ and $n\ge4$, $k>1$ we get $\op{St}^{k+1}_{a_k}=0$.
Thus the action is locally free from the jet-level 2.

Now we can easily compute the orbit dimensions:
the orbit of $\DD^{k+1}_a$ in $\E^k_a$ has dimension of $\E^k_a=J^k_a$ for $k\leq1$ and
it has dimension of $\DD^{k+1}_a$ for $k\ge2$. Below we use the formulae $\dim S^kT^*_a=\binom{n+k-1}{k}$,
$\dim \oplus_{i\leq k}S^iT^*_a=\binom{n+k}{k}$, in particular $\dim\DD^k_a=n\binom{n+k}{k}$.

The Einstein equation $\E$ is expressed by $\binom{n+1}{2}-1$ second-order conditions
(traceless Ricci tensor vanishes). Consider at first the Ricci-flat equation,
whose symbol $\z_{\op{Ric}}$ is resolved via the following acyclic complex (see \cite{Be,K1.5,LY})
 $$
0\to g_k\longrightarrow S^kT^*_a\ot S^2T^*_a \stackrel{\z_{\op{Ric}}}\longrightarrow S^{k-2}T^*_a\ot S^2T^*_a
 \stackrel{\z_{\op{Bnc}}}\longrightarrow S^{k-3}T^*_a\ot T^*_a\to0
 $$
in which $\z_{\op{Bnc}}$ is the symbol of the Bianchi operator and $g_k$ is the symbol of $\E$,
i.e.\ $\op{Ker}(T\E^k_a\to T\E^{k-1}_a)$. Thus for the Ricci flat equation we get
 $$
\dim g_k=\binom{n+k-1}{k}\binom{n+1}{2}-\binom{n+k-3}{k-2}\binom{n+1}{2}+\binom{n+k-4}{k-3}n
 $$
(we let $\binom{m}{k}=0$ for $k<0$).
 % In particular, $\dim g_0=\binom{n+1}2$ and $\dim g_1=n\binom{n+1}2$.
The only difference for Einstein equation is that we change $\dim g_2=\binom{n+1}{2}^2-\binom{n+1}{2}$ to
$\dim\bar{g}_2=\dim g_2+1$.

This implies $h_0=h_1=0$ and $h_2=\dim\bar{g}_2-\dim\Delta_3-\dim SO(g)=\frac1{12}n^2(n^2-13)+1$
(the stabilizer is resolved at this step). For $k>2$ we obtain
$h_k=\dim g_k-\op{dim}\Delta_{k+1}$ and the result follows.
 \end{proof}

This proposition implies the formula
% for $n=4$:: $h_k=0$ for $k<2$, $h_2=4$ and $h_k=2(k+3)(k-1)$ for $k>2$.
 $$
P(z)=\frac{n(z+1)((n+1)z-2(z^2+1))}{2z(1-z)^{n-1}}-\tbinom{n}2(z^2-1)+\frac{n}z+z^2.
 $$
For physically relevant case of Lorntzian metrics in 4D this formula simplifies to the following,
where the first term in the last expression is the Poincar\'e function for Ricci-flat 4D metric derived in \cite{LY}.
 $$
P(z)=\frac{z^2(5+9z-15z^2+5z^3)}{(1-z)^3}=\frac{2z^2(2+6z-9z^2+3z^3)}{(1-z)^3}+z^2.
 $$

% % % % %
\subsubsection{{\bf Self-dual metrics in 4D}}\label{2c}

This is another important special stratum. The self-duality condition
$*W_g=W_g$ is an equation $\E$ on the sections of the bundle $E$ from the general case.

The description and count of differential invariants in this case was done by the author and E.\ Schneider
\cite{KS}: $h_k=0$ for $k<2$, $h_2=9$ and $h_k=\tfrac16(k-1)(k^2+25k+36)$ for $k>2$. This implies the formula
for the Poincar\'e function \cite{KS}:
 $$
P(z)=\frac{z^2(9+4z-30z^2+24z^3-6z^4)}{(1-z)^4} .
 $$

% % % % %
\subsubsection{{\bf K\"ahler metrics on $M^{2n}$}}\label{2d}

Though considered as metrics, they are not a stratum in the space of metrics.
Indeed, a K\"ahler structure is given by the first order equation $\E=\{\nabla^gJ=0\}$ on the bundle of almost
Hermitian pairs $(g,J)$ over $M$, the algebraic constraints are: $J^2=-\1$, $J^*g=g$. The signature of $g$
does not influence the computation below, which thus applies to pseudo-K\"ahler structures as well.

The count of differential invariants in this case was done by A.\ Schmelev
\cite{S}: $h_k=0$ for $k<2$, $h_2=\frac14n^2(n-1)(n+3)$ and
$h_k=\binom{2n+k+1}{k+2}-2\binom{n+k+1}{k+2}-2n\binom{n+k}{k+1}$ for $k>2$.
Note that though dimensions $s_k=\sum_{i=0}^kh_i$ are correct in \cite{S}, the sequence $h_k$
in the Poincar\'e series in Theorem 2.12 of loc.cit.\ has a flaw (wrong coefficient $h_2$ at $z^2$).
The proper formula is (note also that the case $n=1$ is special, since the equation $\E$ is trivial and a
K\"ahler metric is identical to a 2D Riemannian metric and an orientation):
 $$
P(z)=\left\{
\begin{array}{ll}
\frac{z^2(1-z+2z^2-z^3)}{(1-z)^2}, & \text{ for }n=1,\vphantom{\frac{\frac22}{\frac22}}\\
\frac1{z^2(1-z)^{2n}}-\frac{2(zn+1)}{z^2(1-z)^n}+n^2(1-z^2)+\frac{2nz+1}{z^2}, & \text{ for }n>1
\vphantom{\frac{\frac22}{\frac22}}.
\end{array}
\right.
 $$

% Remarks about Shmelev's paper.
% 1. The case $n=1$ is special.
% 2. Formula for the Poincar\'e series for Riemannian metrics in Theorems 1.8, 2.12 and ? is incorrect,
% the term at $z^2$ has a wrong coefficient.
% 3. The series, though rational, are non-explicit.

% {\bf 2c'. K\"ahler-Einstein metrics on $M^{2n}$.} -- unknown...

% \smallskip

% {\bf 2c". Gray-Hervella classes.} -- unknown...

% \medskip

% % % % %
\subsubsection{{\bf Hyper-K\"ahler metrics on $M^{4n}$}}\label{2e}

Similarly, consider the equation $\E$ in the bundle of almost hyper-Hermitian structures $(g,I,J,K)$
given by conditions that the operator fields $I,J,K$ satisfy the quaternionic relations and
are orthogonal with respect to $g$. The equation $\E$ describes integrability of $I_a\in\{I,J,K\}$ and
closedness of the corresponding 2-forms $\omega_a=g(I_a\cdot,\cdot)$; equivalently the condition $\nabla_gI_a=0$
is imposed for all $a$. The pseudo-group is $G=\op{Diff}_{\text{loc}}(M)$, as before.
 % Equivalently, one can study the action of the pseudogroup of holomorphic symplectic transformations on
 % a holomorphic symplectic manifold with a metric $g$ for which both complex and symplectic structures are orthogonal.

The dimensions $s_k$ were computed by A.\ Schmelev \cite{S}. This implies:
$h_k=0$ for $k<2$, $h_2=\frac16n(n+3)(2n-1)(2n+1)$ and
$h_k=2\sum_{i=0}^n\binom{2n+k-i}{k}(n-i)-\binom{2n+k+1}{k+2}-2\binom{n+k+1}{k+2}$ for $k>2$.
However the coefficient $h_2$ at $z^2$ in the Poincar\'e series in Theorem 3.15 of loc.cit.\ is wrong,
so the answer there is not correct.
 % In addition, the formula for the Poincar\'e series is not explicit, and special case $n=1$ is not noticed
The proper formula is
 % (note also that the case $n=1$ is special: Ricci flat self-dual metrics
 $$
P(z)=\frac{2n}{z(1-z)^{2n+1}}-\frac{3}{z^2(1-z)^{2n}}-n(2n+1)(z^2-1)+\frac{4nz+3}{z^2}.
 $$
Note that for $n=1$ hyper-K\"ahler metrics are Ricci-flat self-dual metrics in 4D, so this case is
on an intersection of subsections \ref{2b} and \ref{2c}.

 %%% %%%
\subsection{Linear connections}\label{S33}

These are sections of the affine bundle $E$ associated with the vector bundle $T^*M\ot T^*M\ot TM$ (or $S^2T^*M\ot TM$
for symmetric connections). Note that in general, the torsion $T_\nabla$ of a connection $\nabla$
is a 0-th order invariant, so the action of $G=\op{Diff}_{\text{loc}}(M)$ on $E$ is not transitive.
The bundle $E$ is however algebraic and the conclusion of Theorem \ref{T1} holds true.
 % (but we cannot state anything about Theorem \ref{T3} in this case).
Indeed, scalar differential invariants of 0-th order are rational invariants of the general linear group on the
space of torsion tensors (note that scalar polynomial differential invariants are only constants \cite{GN}).
For $n\ge3$ these also generate invariant derivations, whence a Lie-Tresse generation property
(first order invariants should be used for $n=2$ to get this).

% % % % %
\subsubsection{{\bf General linear connections on $M^n$}}\label{3a}

Since the connections for $n=1$ are all locally equivalent, we assume $n>1$.
The dimensions $h_k$ were computed by T.\ Thomas \cite{T}, see also \cite{D3}:

 \begin{prop}\label{P3a}
For $n>2$ we have: $h_0=\frac12n^2(n-3)$ and $h_k=n^3\binom{n+k-1}{k}-n\binom{n+k+1}{k+2}$, $k>0$.
In the exceptional case $n=2$ we get: $h_0=0$, $h_1=6$, $h_k=6k+2$, $k>1$.
 \end{prop}

The computation is easy, we present a short independent argument.

 \begin{proof}
Let us split $\nabla_XY=\nabla^0_XY+\frac12T_\nabla(X,Y)$, where $T_\nabla\in\Lambda^2T^*_a\ot T_a$ is
torsion of the connection $\nabla$ and $\nabla^0$ is a symmetric connection.
The tensor $T_\nabla$ is a first-order structure, i.e.\ $r=1$, and the action of the group $G=\op{Diff}_{\text{loc}}(M)$
is locally free starting from the jet-level 0 for $n>2$ and starting from the jet-level $1$ for $n=2$.
The symmetric connection $\nabla^0$ has order $r=2$;
it is not a tensor, but a section of an affine bundle with corresponding vector space $S^2T^*_a\ot T_a$.
The action of $G$ is locally free starting from the jet-level 1.
Indeed, the symbolic system associated to the action is the following:
$\g_0=T_a$, $\g_1=\op{End}(T_a)$ and $\g_i=0$ for $i\ge2$. Thus
$\pi_{k+2,3}:\op{St}^{k+2}_{a_k}\stackrel{\sim}\to\op{St}^3_{a_1}=0$ for generic $a_1$.

Therefore in the case $n>2$ we get: $h_0=\dim g_0-\dim\Delta_1-\dim\Delta_2=n^3-n^2-n\binom{n+1}{2}$
and for $k>0$
 $$
h_k=\dim g_k-\dim\Delta_{k+2}=n^3\binom{n+k-1}{k}-n\binom{n+k+1}{k+2}.
 $$
In the case $n=2$ the torsion has a unique non-zero orbit, so $h_0=0$ and the 2-dimensional stabilizer
is resolved in 1-jets, whence $h_1=6$, $h_k=8(k+1)-2(k+3)=6k+2$ for $k>1$.
 \end{proof}

We conclude for general linear connections
(the Poincar\'e function was proved rational in \cite{D3}, but the explicit form was not derived):
 $$
P(z)=\left\{
\begin{array}{ll}
\frac{2z(3+z-z^2)}{(1-z)^2}, & \text{ for }n=2,\vphantom{\frac{\frac22}{\frac22}}\\
\frac{n(n^2z^2-1)}{z^2(1-z)^n}-n^2+\frac{n(nz+1)}{z^2}, & \text{ for }n>2
\vphantom{\frac{\frac22}{\frac22}}.
\end{array}
\right.
 $$

% % % % %
\subsubsection{{\bf Symmetric connections on $M^n$}}\label{3b}

This case, which is a singular stratum in the space of general connections,
was investigated by S.\ Dubrovskiy \cite{D1}. In his computations dimensions of the stabilizers were correctly
determined, but dimensions $h_k$ were computed wrongly due to an arithmetic error (two flaws in Theorem 2.4 of loc.cit.:
factor $n$ before the second binomial coefficient should be omitted, and $n^2$ shall be subtracted from
the coefficient $h_1$). Correcting these yields:

 \begin{prop}
We have: $h_0=0$, $h_1=\frac13\,n^2(n^2-4)+\delta_{2,n}$,
$h_2=n\binom{n+1}{2}^2-n\binom{n+3}{4}-\delta_{2,n}$,
and $h_k=n\binom{n+1}{2}\binom{n+k-1}{n-1}-n\binom{n+k+1}{n-1}$ for $k\ge2$.
 \end{prop}

Let us give some details (our computation is independent of \cite{D1}).

 \begin{proof}
As we noted above, symmetric connections are structures of order $r=2$ and the action of
$G=\op{Diff}_{\text{loc}}(M)$ is locally free starting from the jet-level $1$ for $n>2$. Thus
 $$
h_0=\dim g_0-\dim\Delta_2=n\binom{n+1}{2}-n\binom{n+1}{2}=0.
 $$
Next $h_1=\dim g_1-\dim\Delta_1-\dim\Delta_3=\frac13n^2(n^2-4)$
is the number of rational invariants of the (free) action of the general linear group $GL(n)$ on the space
of curvature tensors $\mathcal{K}=\op{Ker}(\La^2T^*_a\ot\mathfrak{gl}(n)\to\La^3T_a^*\ot T_a)$.
For $k>1$ we get:
 $$
h_k=\dim g_k-\dim\Delta_{k+2}=\frac12n^2(n+1)\binom{n+k-1}{k}-n\binom{n+k+1}{k+2}.
 $$
In the case $n=2$ the action of $GL(2)$ on the space of curvature (or Ricci) tensors is not free,
there is a one-dimensional stabilizer that is resolved on the next jet-level. Thus here
$h_1$ increases by 1 and $h_2$ decreases by $1$, implying the claim.
 \end{proof}

This implies the formula:
 $$
P(z)=\left\{
\begin{array}{ll}
\frac{z(1+5z-z^2-z^3)}{(1-z)^2}, & \text{ for }n=2,\vphantom{\frac{\frac22}{\frac22}}\\
\frac{(n(n+1)z^2-2)n}{2z^2(1-z)^n}-n^2z+\frac{n(1+nz)}{z^2}, & \text{ for }n>2
\vphantom{\frac{\frac22}{\frac22}}.
\end{array}
\right.
 $$

% % % % %
\subsubsection{{\bf Metric connections on $M^n$}}\label{3c}

A metric connection consists of a pair $(g,\nabla)$, where
$g$ is a (pseudo-)Riemannian metric, $\nabla$ a linear connection on $M$, and $\nabla g=0$.
The structure is the pair composed of a metric $g$, which has order $r=1$, and a connection $\nabla$,
which has order $r=2$ with respect to the action of $G=\op{Diff}_{\text{loc}}(M)$.

It is well known that given $g$ such $\nabla$ are bijective with their torsions
$T=T_\nabla\in\Gamma(\La^2T^*M\otimes TM)$.
Indeed, $\alpha_X=\nabla_X-\nabla^g_X$ is a $g$-skew symmetric map for every $X\in TM$.
The skew-symmetrization map $T^*_a\ot\La^2T^*_a\to\La^2T^*_a\ot T^*_a$, given by
$\alpha\mapsto\tau=\alpha_\curlywedge$, $\tau(X,Y,Z)=\frac12(\alpha(X,Y,Z)-\alpha(Y,X,Z))$,
is an isomorphism; its inverse is given by the formula
$\alpha(X,Y,Z)=\tau(X,Y,Z)-\tau(X,Z,Y)-\tau(Y,Z,X)$.
Denote the inverse map $\tau\mapsto\alpha=\tau_\curlyvee$:
$(\tau_\curlyvee)_\curlywedge=\tau$. Using the operation $\sharp$ of raising the indices,
we conclude the formula $\nabla-\nabla^g=\frac12T^\sharp_\curlyvee$, i.e.
$g(\nabla_XY,Z)=g(\nabla^g_XY,Z)+\tfrac12\bigl(g(T(X,Y),Z)+g(T(Z,X),Y)-g(T(Y,Z),X))\bigr)$.

Thus we consider the pairs $(g,T)$, which are simpler objects but their jets are staggered:
$k$-jet of $(g,\nabla)$ corresponds to $(k+1)$-jet of $g$ and $k$-jet of $T$ for $k>0$.
Note that 1-st derivatives of $g$ are bijective with the Christoffel coefficients of $\nabla$.
Indeed, lowering the indices $\Gamma_{ijk}=\Gamma^l_{ij}g_{lk}$, $T_{ijk}=T^l_{ij}g_{lk}$,
one gets the relations $\p_kg_{ij}=\Gamma_{ikj}+\Gamma_{jki}$ and
 $$
2\Gamma_{ijk}=\p_ig_{jk}+\p_jg_{ik}-\p_kg_{ij}+T_{ijk}+T_{jki}-T_{kij}.
 $$
Thus 0-jet of our pair is given by the values of $g$ and $T$, while for $k>0$ the $k$-jet of the pair
is bijective with $(k+1)$-jet of $g$ and $k$-jet of $T$.

Note that the bundle of metric connections is another instance where the pseudogroup $G$ acts non-transitively.
However the same remarks as in subsection \ref{3a} apply here, and we can proceed as before.
The count of invariants for the pairs $(g,T)$ is given as follows.
 \begin{prop}
We have: $h_0=\frac{n}2(n-1)^2$ and
$h_k=\dim g_k-\dim\Delta_{k+1}=\frac{n(n+1)}2\binom{n+k}{k+1}+\frac{n^2(n-1)}2\binom{n+k-1}{k}-n\binom{n+k+1}{k+2}$
for $k>0$.
 \end{prop}

 \begin{proof}
The structure $(g,T)$ is a section of a tensor bundle of rank $\frac12n(n^2+1)$. The general linear group acts
on an open set of it freely (equivalently: the orthogonal group acts freely on the general stratum of
the space of torsions) for $n>1$. Consequently, by the Corollary, the stabilizer $\op{St}^{k+1}_{a_k}$ vanishes
for all $k\ge0$, and so the action is (locally) free from the jet-level 0.
The symbol $g_k$ of the structure, as a staggered pair $(g,T)$, is
$S^{k+1}T^*_a\ot S^2T^*_a\oplus S^kT^*_a\ot\La^2T^*_a\ot T_a$ for $k\ge1$.

Therefore $h_k=\dim g_k-\dim\Delta_{k+1}=\frac{n(n+1)}2\binom{n+k}{k+1}+\frac{n^2(n-1)}2\binom{n+k-1}{k}-n\binom{n+k+1}{k+2}$,
and the claim follows.
 \end{proof}

This implies (again we exclude the case $n=1$ as trivial: $P(z)=0$) the formula for $n\ge2$:
 $$
P(z)=\frac{n-\binom{n}2(z^2-z)}{z^2}-\frac{2n-n(n+1)z-n^2(n-1)z^2}{2z^2(1-z)^n}.
 $$

% % % % %
\subsubsection{{\bf Metric connections with a skew-symmetric torsion}}\label{3c+}

These form a partial case of general metric connections: these consist of $(g,T)$ with $T_{ijk}=T^l_{ij}g_{lk}$
being skew-symmetric in all indices. Since such $T$ vanishes in dimension $n=2$ (this case belongs
to metrizable connections discussed below), so assume $n\ge3$.

On the level of 0-jets, the stabilizer of generic point $a_0\in\E^0_a$ corresponds to stabilizer of a generic
3-form in the orthogonal group, so the sequence $\op{st}_n=\dim\op{St}^2_{a_0}$, depending on $n=\dim M$,
is the following: $\op{st}_3=3$, $\op{st}_4=3$, $\op{st}_5=2$, and $\op{st}_n=0$ for $n\ge6$.

 \begin{rk}
The stabilizers of a generic 3-form in the group $O(g)$ for a fixed metric $g$ on $T_a$, are:
$SO(g)$ for $n=3$,
$SO(3)\times\Z_2$ or $SO(1,2)\times\Z_2$ depending on the signature of $g$ for $n=4$,
$S(O(2)\times O(2))$ for $n=5$,
and trivial for $n\ge6$.
Note that extending $O(g)$ to $GL(n)$ the stabilizers become nontrivial for $n$ up to $8$;
for example, when $n=7$ the stabilizer of a generic 3-form $\omega$ is a real form of
the exceptional Lie group $G_2$; it preserves some metric, depending on $\omega$, but not the given metric $g$.
 \end{rk}

Thus the action is free from the level of 0-jets for $n\ge6$ and from the level of 1-jets for $3\leq n\leq 5$.
This implies
 $$
P(z)=\frac{n-\binom{n}2(z^2-z)}{z^2}-\frac{n-\binom{n+1}2z-\binom{n}3z^2}{z^2(1-z)^n}+\op{st}_n(1-z).
%P(z)=\frac{n}{z^2}-\binom{n}{2}\frac{z-1}z-\frac{n-\binom{n+1}2z-\binom{n}3z^2}{z^2(1-z)^n}+\op{st}_n(1-z).
 $$

The space of metrizable connections consists of such $\nabla$ that $T_\nabla=0$ and there exists a
parallel metric $g$. Generic connections of this type have irreducible holonomy, and for them the metric
$g$ is unique up to scale, i.e.\ $\nabla=\nabla^g$ is the Levi-Civita connection
(for generic $\nabla$ the scale can be fixed by the requirement $\|R_\nabla\|_g^2=\pm1$).
Thus for metrizable connections the Hilbert function is equal to $h_{k+1}$ of subsection \ref{2a},
and hence the Poincar\'e function of our problem is expressed via the function $P(z)$ of subsection \ref{2a}
as $P(z)/z$.

% (with skew-torsion)
% We will assume $n\ge3$ (the case $n=2$ is covered below in the class of metrizable connections).

% % % % %
\subsubsection{{\bf Symplectic connections on $M^{2n}$}}\label{3d}

Similarly, let us count moduli of Fedosov structures $(\omega,\nabla)$, consisting of
a symplectic form $\omega$ on $M$ and a linear symmetric connection $\nabla$ such that $\nabla\omega=0$
(note that this condition and $d\omega=0$ imply $T_\nabla=0$ \cite{V}).
The pairs $(\omega,\nabla)$ with this differential relation form an equation $\E$. In \cite{D2} S.\ Dubrovskiy
investigated the number of differential invariants in the "staggered" jet-filtration
$j^k(\omega,\nabla)=(j^k\omega,j^{k-1}\nabla)$ on $\E$.
In contrast, we consider here the natural jet-filtration $j^k(\omega,\nabla)=(j^k\omega,j^k\nabla)$.

 \begin{prop}
The orbit dimensions for jets of Fedosov structures are: $h_0=0$,
$h_1=\frac12(n-1)n(2n+1)(2n+3)+\delta_{n,1}$, $h_2=\frac15n(n+1)(3n+2)(4n^2-1)-\delta_{1,n}$,
and $h_k=\binom{2n+2}{3}\binom{2n+k-1}{k}-\binom{2n+k+2}{k+3}$ for $k\ge3$.
 \end{prop}

 \begin{proof}
This structure $(\omega,\nabla)$ is of mixed orders, similarly to the metric case: $r=1$ for $\omega$ and
$r=2$ for $\nabla$. It is worth fixing the stabilizer of $\omega$ (since this has the Darboux normal form) to be
the (infinite-dimensional) pseudo-group of symplectomorphisms $G=\op{Diff}_{\text{loc}}(M,\omega)$,
with the Lie algebra sheaf $\mathcal{G}$ whose elements $X_H$ are encoded by Hamiltonians defined up to
constant terms $H\in C^\infty(M^{2n})/\R$.

The condition on the connection to be symplectic means that $\Gamma_{ijk}=\oo_{ia}\Gamma^a_{jk}$ is symmetric
in all indices in coordinates where $\omega$ is constant \cite{D2}. Equivalently, if we fix one symplectic connection
$\nabla^0$, then any other symplectic connection is $\nabla=\nabla^0+A$, where $A\in S^2T_a^*\ot T_a$ is fully
symmetric upon $\omega$-lowering the indices: $\omega(A(\cdot,\cdot),\cdot)\in S^3T^*_a$.

Thus, in the reduced form ($\omega$ fixed, $\nabla$ varies),
the geometric structure is given by a section of an affine
bundle with the corresponding vector bundle $S^3T_a^*$.
Hence the symbols of the equation $\E$ are equal to $g_k\simeq S^kT_a^*\ot S^3T_a^*$.
 % $\dim g_k=\binom{2n+2}{3}\binom{2n+k-1}{k}$.
The action of group $G$ has order $r=3$: the elements depend on 1-jet of the Hamiltonian $H$ and
the lift to the space of connections adds two orders.
 % Equivalently for its Lie algebra: the Hamiltonian field $X_H$ depends on 1-jet of $H$ and we assume
 % $H(a)=0$, $d_aH=0$ at the point $a\in M$, where the action of $G_a$ on $\E_a$ is considered.

The stabilizer of a 0-jet is the symplectic group $\op{St}^2_{a_0}=\op{Sp}(T_a,\omega)$ with
the Lie algebra generated by Hessians $d^2_aH$, and we normalize $H(a)=0$, $d_aH=0$ .
For $n>1$ this stabilizer is resolved on the level of 1-jets, because the linear symplecic group acts freely
on the space of curvatures of $\nabla$ (this space is described in \cite{V}). Thus $h_0=\dim g_0-\dim\Delta_3=0$, $h_1=\dim g_1-\dim\Delta_2-\dim\Delta_4$
and $h_k=\dim g_k-\dim\Delta_{k+3}=0$ for $k\ge2$ as was claimed.

In the case $n=1$ ($\dim M=2$) the curvature is expressed by the Ricci part and the stabilizer is resolved
on the next jet-level, i.e.\ the action is locally free not from the level of 1-jet, but from the level of 2-jets.
Again $h_0=0$, but now $h_1=1$ -- the only invariant of order one is the norm of the Ricci tensor of the connection
with respect to $\oo$. Thus a 1-dimensional stabilizer exists on this level, but it resolves on the next level,
and we have $h_2=5$, $h_k=3k$ for $k\ge3$. In other words, with respect to the formulae for $n>1$ dimension $h_1$
increases by 1 and $h_2$ decreases by 1, the other dimensions being un-changed.
 \end{proof}

This implies the formula for the Poincar\'e function: % in \cite{D2} it is clearly different
 $$
P(z)=\left\{
\begin{array}{ll}
\frac{z(1+3z-z^3)}{(1-z)^2}, & \text{ for }n=1,\vphantom{\frac{\frac22}{\frac22}}\\
\frac{2n(2n^2+3n+1)z^3-3}{3z^3(1-z)^{2n}}+\frac{1+2nz-n(2n+1)z^2(z^2-1)}{z^3}, & \text{ for }n>1
\vphantom{\frac{\frac22}{\frac22}}.
\end{array}
\right.
 $$

% % % % %
\subsubsection{{\bf Projective connections on $M^n$}}\label{3e}

Two linear connections are projectively equivalent
if their geodesics coincide as un-parametrized curves. An equivalence class is called a projective
connection. Every such structure is represented by a symmetric connection, and two symmetric connections
$\nabla$, $\nabla'$ are projectively equivalent iff for some 1-form $\nu$ we have:
 $$
\nabla_XY-\nabla'_XY=\nu(X)Y+\nu(Y)X.
 $$
In components, an equivalence class is represented by Thomas' symbols
$\Pi^k_{ij}=\Gamma^k_{ij}-\frac1{n+1}(\delta^k_i\Gamma^a_{aj}+\delta^k_j\Gamma^a_{ai})$, which is the
traceless symmetric part of the Christoffel symbol $\Gamma^k_{ij}$.

The number of differential invariants (projective scalars) was computed in \cite{L2}
(2D projective connections are equivalent to cubic ODE considered by Lie and Tresse, see subsection \ref{1b}).
We provide an independent short computation.

 \begin{prop}
For $n>2$ we get $h_0=0$, $h_1=\frac13n^2(n^2-7)$, $h_2=\frac{n}{24}(n-2)(5n^3+16n^2+15n+12)$,
$h_k=\frac{n}2(n-1)(n+2)\binom{n+k-1}{k}-n\binom{n+k+1}{k+2}$ for $k>2$.
For $n=2$, $h_k=0$ for $k<4$ and $h_k=2(k-1)$ for $k\ge4$.
 \end{prop}

 \begin{proof}
These structures are sections of an affine bundle of rank $n\binom{n+1}2-n=\frac12(n-1)n(n+2)$, whence
$\dim g_k=\frac12(n-1)n(n+2)\binom{n+k-1}{k}$.

The symbolic system associated to the action is $\g_0=T_a$,
$\g_1=\op{End}(T_a)$, $\g_2=\g_1^{(1)}\simeq T_a^*$ and $\g_i=0$ for $i>2$. This implies that
$\op{St}^{k+2}_{a_k}$ stabilize from the level $k=2$, and for generic $a_2$ this stabilizer vanishes.
Thus the action is locally free from the jet-level 2, and hence $h_k=\dim g_k-\dim\Delta_{k+2}$ for $k\ge3$.

This formula is modified in lower orders as follows: $h_0=0$ (no 0$^\text{th}$ order invariants;
increase of $h_0$ by $n$ with respect to the general formula), $h_1=\dim g_1-\dim\Delta_1-\dim\Delta_3$
(1$^\text{st}$ order invariants are obtained from the curvature tensor through quotient by the group
of time reparametrizations and the general linear group), $h_2=\dim g_2-\dim\Delta_4+n$ (resolution
of the stabilizer from the level of 0-jets).
 % decrease of $h_2$ by $n$ with respect to the general formula).

This can be also justified by direct rank computation for a system of vector fields \cite{L2}.
The case $n=2$ is special: the action becomes locally free only starting from the jet-level 3,
and the additional 4-dimensional stabilizer on the level of 1-jets is resolved on the level of 3-jets.
 \end{proof}

This computation implies the formula:
 $$
P(z)=\left\{
\begin{array}{ll}
\frac{2z^4(3-2z)}{(1-z)^2}, & \text{for }n=2,\vphantom{\frac{\frac22}{\frac22}}\\
\frac{n}{(1-z)^n}\bigl(\binom{n+1}{2}-\tfrac{1+z^2}{z^2}\bigr)-
n\bigl(z^2+nz-1-\tfrac{nz+1}{z^2}\bigr), & \text{for }n>2\vphantom{\frac{\frac22}{\frac22}}.
\end{array}
\right.
 $$

 %%% %%%
\subsection{Conformal an related structures on $M^n$}\label{S34}

A conformal structure is a metric up to re-scaling by a positive
function, and so it is a section of the bundle $(S^2T^*M\setminus0_M)/\R_+$
with a non-degenerate representative at every point.
For $n=1,2$ all metrics are conformally flat, so to get local invariants we restrict to $n\ge3$.

% % % % %
\subsubsection{{\bf General conformal structures}}\label{4a}

As proven in \cite{L1,K3} in the case $n>3$, $h_0=h_1=0$, $h_2=\frac{n^2(n^2-1)}{12}-n^2-1$,
$h_3=\frac1{24}n(n^4+2n^3-5n^2-14n-32)$ and
$h_k=(\binom{n+1}{2}-1)\cdot\binom{n+k-1}{k}-n\cdot\binom{n+k}{k+1}=\frac{n(k-1)}2\binom{n+k-1}{k+1}-\binom{n+k-1}{k}$ for $k\ge4$. In the case $n=3$, $h_0=h_1=h_2=0$, $h_3=1$, $h_4=9$ and $h_k=k^2-4$ for $k\ge5$. This yields \cite{K3}:
 $$
P(z)=\left\{
\begin{array}{ll}
%\frac{z^3(1+6z-3z^2-5z^3+3z^4)}{(1-z)^3}, & \text{ for }n=3,\vphantom{\frac{\frac22}{\frac22}}\\
\frac{z^3(1+z)(1+5z-8z^2+3z^3)}{(1-z)^3}, & \text{ for }n=3,\vphantom{\frac{\frac22}{\frac22}}\\
\frac{(n+1)nz-2(n+z)}{2z(1-z)^n}+\frac{n}z+
\bigl(1+\tbinom{n}{2}+nz\bigr)(1-z^2), & \text{ for }n>3
\vphantom{\frac{\frac22}{\frac22}}.
\end{array}
\right.
 $$

% % % % %
\subsubsection{{\bf Weyl conformal structures on $M^n$}}\label{4b}

A Weyl structure is a pair consisting of a conformal structure $[g]$ and a linear connection $\nabla$ preserving it.
In terms of the representative $g$, this means $\nabla g=\omega\otimes g$ for a 1-form $\omega$ on $M$.
Conformal re-scaling of the representative $g\mapsto e^f g$ results in the shift of Weyl potential
$\omega\mapsto\omega+df$. The resulting equivalence class of pairs $(g,\omega)$ is often considered
as a Weyl structure, and we follow this agreement.

Note that $k$-jet of $(g,\omega)$ is equivalent to $k$-jet of $g$ and $(k-1)$-jet of $\nabla$.
Indeed, for the Levi-Civita connection of $g$ we have:
$\nabla-\nabla^g=\omega^\sharp\in S^2T^*_a\otimes T_a$,
where in terms of the Christoffel symbols $\gamma$ of $g$ and $\Gamma$ of $\nabla$,
the tensor $\omega^\sharp$ is given in coordinates as
 $$
\Gamma_{ij}^k-\gamma_{ij}^k=(\omega^\sharp)^k_{ij}=\tfrac12(\omega_i\delta_j^k+\omega_j\delta_i^k-g_{ij}\omega^k).
 $$

However $k$-jet of $\nabla$ yields $(k-1)$-jet of $R_\nabla$ and, by taking the skew-part of Ricci,
it yields $(k-1)$-jet of $d\omega$; using the freedom in shifting the Weyl potential by $df$, this gives
$k$-jet of $\omega$ and hence $k$-jet of $\nabla^g$ and (cf.\ subsection \ref{3c})
$(k+1)$-jet of $g$, provided that 0-jet of $g$ is known.

Thus $k$-jet of the Weyl structure $([g],\nabla)$, for $k>0$, is equivalent to the staggered jet, consisting
of $(k+1)$-jet of the metric $g$ and $k$-jet of the Weyl potential $\omega$ modulo the equivalence
$(g,\omega)\simeq(e^fg,\omega+df)$. This is the filtration we will be using in our pseudogroup orbit study.
Note that the action of $G=\op{Diff}_{\text{loc}}(M)$ on $([g],\nabla)$ has order $r=1$ in the first component
and order $r=2$ in the second component.

 \begin{prop}
We have: $h_0=0$, $h_1=\frac1{12}(n^2-4)(n^2+3)+\delta_{2,n}$, $h_2=\frac1{24}n(n^2-1)(n^2+2n+8)-\delta_{2,n}$, and
$h_k=(\binom{n+1}{2}-1)\binom{n+k}{k+1}+n\binom{n+k-1}{k}-n\binom{n+k+1}{k+2}$ for $k>2$.
 \end{prop}

 \begin{proof}
It is easy to see that the group $\Delta_1\oplus\Delta_2$ acts transitively on
the 0-jets $([g],\Gamma)$. The stabilizer is $\op{St}^2_{a_0}=CO(n)\subset\Delta_1$.
The action of $\Delta_3$ on 1-jet of $\nabla\equiv\Gamma$ is free and the previous stabilizer $CO(n)$
is resolved upon the action on the space of curvature tensors $\{R_\nabla\}$, so that $\op{St}^3_{a_1}=0$ for
$n>2$; in the case $n=2$ the scalar part of $CO(2)=\R^*\times SO(2)$ is reduced to $\Z_2$, while the
rotation persists: $\op{St}^3_{a_1}=O(2)$; both are resolved on the next jet-level: $\op{St}^4_{a_2}=0$.

Thus the action is (locally) free from the level of 1-jets for $n>2$ and $2$-jets for $n=2$.
This allows computing the counting function $h_k$. Indeed, the symbol space $g'$ for conformal structures
satisfies $\dim g'_k=(\binom{n+1}{2}-1)\binom{n+k-1}{k}$, and the symbol space $g''$ for Weyl potentials (the
scaling factor $f$ is taken into consideration when counting $g'_k$) satisfies $\dim g''_k=n\binom{n+k-1}{k}$.
Thus for $k=1$ we compute $h_1=\dim g'_2+\dim g''_1-\dim\Delta_3-\dim CO(n)$; note also that this equals
the dimension of the space of the curvature tensors of $\nabla$ (counting also the skew-part of Ricci)
mod stabilizer group action: $h_1=\frac{n^2(n^2-1)}{12}+\binom{n}2-\dim CO(n)=\frac1{12}(n^2-4)(n^2+3)$.
For $k>1$ we get $h_k=\dim g'_{k+1}+\dim g''_k-\dim\Delta_{k+2}$. In the case $n=2$ the number $h_1$ shall be
increased by 1 and $h_2$ decreased by 1.
 \end{proof}

This computation yields for $n>1$ the formula:
 $$
P(z)=\frac{nz^2+(\binom{n+1}{2}-1)z-n}{z^2(1-z)^n}-\frac{(\binom{n}{2}+1)z(z^2-1)-n}{z^2}-\delta_{2,n}z(z-1).
 $$

% % % % %
\subsubsection{{\bf Einstein-Weyl structures on $M^n$}}\label{4c}

The Einstein-Weyl condition is the following set of $\binom{n+1}2-1$ equations $\op{Ric}^{\op{sym}}_\nabla=\Lambda g$,
where $\Lambda=\frac1n\op{Tr}_g\op{Ric}^{\op{sym}}_\nabla$, on the unknown Weyl structure $[(g,\omega)]$.
This condition is vacuous for $n=2$, so we assume $n>2$ for this structure.

 \begin{prop}
We have: $h_0=0$, $h_1=\frac1{12}(n-3)n(n+1)(n+2)+\delta_{n,3}$,
$h_2=\frac1{24}n(n-1)(n-2)(n^2+5n+8)-\delta_{n,3}$, and
$h_k=(\binom{n+1}{2}-1)\binom{n+k}{k+1}+n\binom{n+k-1}{k}-(\binom{n+1}{2}-1)\binom{n+k-2}{k-1}-n\binom{n+k+1}{k+2}$
for $k>2$.
 \end{prop}

 \begin{proof}
There are two important specifications in this case. First, the
structures are given by a differential system $\E$ on $[(g,\omega)]$. It consists of $\binom{n+1}{2}-1$ equations
of the second order. This system is determined (not as it stands, because it has more dependent variables
$\binom{n+1}{2}+n-1$ than the equations, but determinacy comes modulo the diffeomorphism freedom; see \cite{DFK}
for an effective quotient in the arguably most important case $n=3$), so its prolongation will
have $r_k=(\binom{n+1}{2}-1)\binom{n+k-1}k$ equations of order $k+2$ on $g$.

Second, we still have $\op{St}^2_{a_0}=CO(n)$ but the stabilizer $\op{St}^3_{a_1}$ changes.
While for $n>3$ the curvature of $\nabla$ contains the Weyl tensor as an irreducible part (reducible
into anti- and self-dual parts for $n=4$) and the action of $CO(n)$ resolves on it,
in the case $n=3$ the curvature of $\nabla$ consists of the trace part of $\op{Ric}_\nabla^{\text{sym}}$
(due to Einstein-Weyl condition) and $\op{Ric}_\nabla^{\text{skew}}\equiv d\omega$ thus reducing
$\op{St}^2_{a_0}=CO(3)$ to $\op{St}^3_{a_1}=O(2)$, and in the next jet-level this stabilizer is also resolved.
Thus we conclude $\dim\op{St}^3_{a_1}=\delta_{n,3}$ and $\dim\op{St}^{k+2}_{a_k}=0$ for $k\ge2$.

These two observations imply: $h_1=\frac{n^2(n^2-1)}{12}+\binom{n}2-r_0-\dim CO(n)=\frac1{12}(n-3)n(n+1)(n+2)$.
For $k>1$ we get: $h_k=\dim g'_{k+1}+\dim g''_k-r_{k-1}-\dim\Delta_{k+2}$. In the case $n=3$ the number $h_1$ shall be
increased by 1 and $h_2$ decreased by 1, implying the claim.
 \end{proof}

This computation yields for $n>2$ the formula:
 $$
P(z)=
\left\{
\begin{array}{ll}
\frac{z(1+5z-z^2-z^3)}{(1-z)^2}, & \text{ for }n=3,\vphantom{\frac{\frac22}{\frac22}}\\
\frac{nz^2-(\binom{n+1}{2}-1)z(z^2-1)-n}{z^2(1-z)^n}-\frac{(\binom{n}{2}+1)z(z^2-1)-n}{z^2},
& \text{ for }n>3 \vphantom{\frac{\frac22}{\frac22}}.
\end{array}
\right.
 $$

 \begin{rk}
For $n=3$ computation of the Poincar\'e function in a different jet-filtration for Weyl and Einstein-Weyl
structures was done recently in \cite{KS2}. The results differ from the above, but agree in asymptotic.
This is an effect of ``staggering'' jets or ``normalizing'' the structure. % via integration.
 \end{rk}

% % % % %
\subsubsection{{\bf Self-dual conformal structures in 4D}}\label{4d}

Self-duality equation $*W_g=W_g$ for the Weyl tensor of metric $g$ (of Riemannian or neutral signature)
has meaning only in dimension 4. The count for differential invariants for self-dual conformal structures $[g]$
on $M^4$ was performed in \cite{KS}: $h_0=h_1=0$, $h_2=1$, $h_3=13$,
$h_k=3k^2-7$ for $k>3$ (note the difference with subsection \ref{2c}, where we considered $g$ intead of $[g]$).
Consequently the Poincar\'e function is \cite{KS}:
 $$
P(z)=\frac{z^2(1+10z+5z^2-17z^3+7z^4)}{(1-z)^3}.
 $$

 %%% %%%
\subsection{Almost complex structures on $M^{2n}$}\label{S35}

This structure of order $r=1$ is given by a field $J\in\op{End}(TM)$ with $J^2=-\bf{1}$.
This is the first non-trivial example of an infinite type geometric structure, meaning that its symbol
allows infinite-dimensional symmetry algebra, which is realized for the standard (integrable) complex
structure on $\C^n$, though generic almost complex structures have no local symmetry at all \cite{K0.1}.

In more details, $J$ is a $G$-structure with the group $G=\op{GL}(n,\C)$, whose Lie algebra $\g=\mathfrak{gl}(n,\C)$
has infinite type: its prolongation is the algebra $S(\C^n)^*\ot_\C\C^n$ of formal holomorphic vector fields at $0$.
However the prolongation-projection of the Lie equation for $J$ encodes conservation of both $J$ and its
Nijenhuis tensor $N_J$. For $n>2$ this is already a finite type structure in general. But for $n=2$ it is
still of infinite type, and one has to do one more prolongation-projection to achieve finite type.

 \begin{theorem}
For almost complex structure the count of invariants is as follows: $h_0=0$ and
$h_k=2n^2\binom{2n+k-1}k-2n\binom{2n+k}{k+1}+2n\binom{n+k}{k+1}-2n\binom{n+k-1}{k}+2(\delta_{k,1}-\delta_{k,2})\delta_{n,3}$,
for $k>0$, $n\ge3$. In the case $n=2$, $h_0=h_1=0$, $h_2=2$ and $h_k=8\binom{k+3}k-4\binom{k+4}{k+1}+4$.
 \end{theorem}

 \begin{proof}
We will do this computation in several steps.

First note that $\op{St}^1_{a_0}\simeq G=\op{GL}(n,\C)$.
Since $J$ is a $G$-structure with $\g=\op{Lie}(G)\subset\op{End}(T_a)$,
Proposition \ref{pLY} shows that $\op{St}^{k+1}_{a_k}\supset\g_{k+1}=\mathfrak{g}^{(k)}=S^{k+1}_\C T^*_a\ot_\C T_a$.
The latter space is
 $$
\{\Phi:S^{k+1}T_a\to T_a,\Phi(X_0,\dots,JX_i,\dots,X_k)=J\Phi(X_0,\dots,X_i,\dots,X_k)\}
 $$
and has (real) dimension $2n\binom{n+k}{k+1}$.
The analog of freeness from the jet-level $k$ is the equality in the previous inclusion.

The equation for symmetries of $J$ has in 1-prolongation the condition that $N_J$ is preserved.
 % In other words, the prolongation-projection of $\E$ (with $\E_2=\E_1^{(1)}$ etc) is the equation $\tilde\E$
 % with $\tilde\E_1=\pi_{2,1}(\E_2)=\{[\vp]_a^1:d_a\vp:T_aM\to T_{\vp(a)}M\text{ preserves both }J\text{ and }N_J\}$.
Consequently, the corresponding symbolic system
 $$
\tilde{\mathfrak{g}}=\{\Phi:T_a\to T_a:J\Phi=\Phi J,N_J(\Phi\cdot,\cdot)+N_J(\cdot,\Phi\cdot)=\Phi N_J(\cdot,\cdot)\}
 $$
is of finite type for generic $N_J$ and $n\ge3$. Indeed, consider the bundle $\tilde\E$ consisting of the pair $J$
and its Nijenhuis tensor $N_J$, namely $\tilde{a}_0=(J,N_J)$. The tensor $N_J$ involves 1-jet of $J$,
yet the action of the pseudogroup $\op{Diff}_\text{loc}(M)$ on $\tilde\E$ still has order $r=1$.
Let $\widetilde{\op{St}}^{k+1}_{\tilde{a}_k}\subset\DD^{k+1}_a$ denote the corresponding stabilizers.
 \begin{lem}
For $n>3$ and generic $\tilde{a}_0$ (i.e.\ generic 1-jet of $J$) we have:
 $$
\widetilde{\op{St}}^1_{\tilde{a}_0}=\{\Phi\in\op{End}(T_a,J):N_J(\Phi\cdot,\Phi\cdot)=\Phi N_J(\cdot,\cdot)\}=\Z_3.
 $$
 \end{lem}

 \begin{proof}
Identifying $(T_a,J)$ with $\C^n$ observe that scaling by $e^{2\pi i/3}$ always belongs to
$\widetilde{\op{St}}^2_{\tilde{a}_0}$. To show that generically there are no other symmetries, note that
the subgroup $\widetilde{\op{St}}^2_{\tilde{a}_0}\subset\op{GL}(n,\C)$ is upper-semicontinuous in
$N_J\in\La^2(\C^n)^*\otimes_{\bar\C}\C^n$, so if we show the claim for one $N_J$ it will follow for a
Zariski generic element as well.

Consider the following element given in terms of a complex basis $e_1,\dots,e_n\in\C^n$:
 \begin{multline*}
N_J(e_1,e_k)=e_{k+1}\ (2\le k<n),\ N_J(e_1,e_n)=e_2,\\
N_J(e_2,e_3)=e_1,\ N_J(e_2,e_k)=ke_k\ (3<k\leq n).
 \end{multline*}
A moment of thought shows that the only complex transformations preseving this $N_J$ are diagonal, i.e.\
$\Phi(e_k)=\rho_ke^{i\theta_k}e_k$, where $\rho_k\in\R_+$ and $\theta_k\in\R\!\!\mod\!2\pi$. The first line
of the defining relations yields $\rho_1=1$, $\rho_2=\dots=\rho_n$, and then we get that all $\rho_k=1$.
After this it is easy to obtain $\theta_k=\theta\in\{0,2\pi/3,4\pi/3\}$.
 \end{proof}
Thus, for $n>3$ we have $\tilde{\mathfrak{g}}_1=0$, and hence $\tilde{\mathfrak{g}}_k=0$ for $k>1$.

For $n=3$ the normal forms of \cite{K0.1} yield $\dim\tilde{\mathfrak{g}}_1=2$ for generic $N_J$ and a
straightforward computation shows that $\tilde{\mathfrak{g}}_2=0$.
 % For example, in the NDG1 case with generic parameters, where the Nijenhuis tensor is given by
 % $N_J(X_1,X_2)=X_2$, $N_J(X_1,X_3)=\lambda X_3$, $N_J(X_2,X_3)=e^{i\varphi}X_1$,
 % the stabilizer is $\g_1=\{\Phi(X_1,X_2,X_3)=(e^{2i\theta}X_1,\rho e^{-i\theta}X_2,\rho^{-1}e^{-i\theta}X_3)\}$;
 % its prolongation $\g_2=\g_1^{(1)}$ is trivial.
 % The similar conclusion happens in NDG3 case, and these are generic NDG cases.

 \begin{prop}\label{acs2}
In the case $n>3$ we have $\op{St}^k_{a_{k-1}}=\g_k$ for $k>1$ and generic $a_{k-1}\in\E^{k-1}$.
When $n=3$ this equality holds true as well except for $k=2$, in which case
we have $\dim\op{St}^2_{a_1}/\g_2=2$ for generic $a_1\in\E^1$.
  \end{prop}

 \begin{proof}
If a diffeomorphism $\vp$ preserves the $k$-jet of $J$, then it preserves $(k-1)$-jet of $(J,N_J)$.
Hence an injective map $\op{St}^{k+1}_{a_{k}}/\g_{k+1}\to \widetilde{\op{St}}^k_{\tilde{a}_{k-1}}$
(in fact, an isomorphism). These can be united into a commutative diagram
 $$\begin{CD}
\op{St}^{k+1}_{a_{k}}/\g_{k+1} @>>> \op{St}^{k}_{a_{k-1}}/\g_{k}\\
 @VVV @VVV \\
\widetilde{\op{St}}^k_{\tilde{a}_{k-1}} @>>>  \widetilde{\op{St}}^{k-1}_{\tilde{a}_{k-2}}
 \end{CD}$$
By Proposition \ref{pLY} the kernel of the bottom map is
$\widetilde{\op{St}}^k_{\tilde{a}_{k-1}}\cap\Delta_k=\tilde\g_k$, and so
by induction $\widetilde{\op{St}}^k_{\tilde{a}_{k-1}}=0$ for $n>3$, implying the first claim.

For $n=3$ the upper arrow of the diagram is injective. Moreover it can be directly checked
(for instance, via the normal forms of \cite{K0.1}) that
$\widetilde{\op{St}}^2_{\tilde{a}_1}=0$, and so $\op{St}^{k+1}_{a_k}=\g_{k+1}$ for $k\ge2$.
The same equality fails for $k=1$. A straightforward computation in Maple gives $\dim\op{St}^2_{a_1}=38$,
while $\dim\g_2=36$, implying the second claim.
 \end{proof}

The space of all almost complex structures $\E$ is the fiber bundle with fiber $F=GL(2n,\R)/GL(n,\C)$
of dimension $2n^2$, and the $k$-symbol is $g_k=S^kT^*M\ot TF$. Consequently, the number of pure order $k>0$
differential invariants for $n>3$ is (for $k=0$ we have $h_0=0$):
 \begin{multline*}
h_k=\dim g_k-\dim\Delta_{k+1}+\dim\g_{k+1}-\dim\g_k\\
=2n^2\binom{2n+k-1}k-2n\binom{2n+k}{k+1}+2n\binom{n+k}{k+1}-2n\binom{n+k-1}{k}.
 \end{multline*}
For $n=3$ we have modification $h_1=2$, $h_2=64$, i.e.\ $h_k\mapsto h_k+2(\delta_{k,1}-\delta_{k,2})\delta_{n,3}$
due to existence of 2 scalar invariants of order 2 \cite{K0.1}.
 % The differential invariants have been also discussed in \cite{K0.2}.

In the case $n=2$ the situation is more complicated: $N_J$ is encoded by $J$-invariant
subspace ($\C$-line) $\Pi_a=\op{Im}N_J\subset T_a$ and an element of a $\C$-line
$(T_a/\Pi_a)^*\ot_\C\op{End}_{\bar\C}(\Pi_a)$.
Thus the system $\tilde\g$, obtained by prolongation-projection is not of finite type:
$\tilde\g_{k+1}=\tilde\g^{(k)}$ satisfies $\dim\tilde\g_{k+1}=2$ for all $k>0$, while $\dim\tilde\g_1=4$.
Indeed, we have in a complex basis $X_1\in\Pi_a$, $X_2\in T_a\setminus\Pi_a$ of $T_a$:
 $$
\tilde\g=\left\{\Phi:(T_a,J)\to(T_a,J)\,|\,
\Phi=\begin{pmatrix}\rho-i\theta & b \\ 0 & 2i\theta\end{pmatrix},\rho,\theta\in\R,b\in\C\right\}.
 $$
Therefore the prolongation $\tilde\g_k\subset S^k_\C T^*_a\ot_\C T_a$ consists of elements
$\Phi_k$ with $i_v\Phi_k=0$ $\forall v\in\Pi_a$ and $\op{Im}\Phi_k\in\Pi_a$. In other words, for $k>1$
$\Phi_k\in S^k_\C(T_a/\Pi_a)\ot_\C\Pi_a$ and the latter space has real dimension 2.

The next prolongation-projection is encoded by a complete parallelism,
namely the points of $\tilde{\tilde\E}$ are frames related to $(J,N_J,[\Pi,\Pi])$ \cite{K0.1}
(so determined by the 2-jet of $J$), whence $\tilde{\tilde\g}_k=0$ for $k>0$.
 \begin{prop}\label{acs3}
For $n=2$ and $k\ge3$ the map $\op{St}^{k+1}_{a_{k}}/\g_{k+1}\to\op{St}^{k}_{a_{k-1}}/\g_{k}$ is an isomorphism
and the spaces have dimensions 2.
 \end{prop}

 \begin{proof}
By the argument from the proof of Proposition \ref{acs2} we have an injective map
$\widetilde{\op{St}}^{k+1}_{\tilde{a}_k}/\tilde\g_{k+1}\to\widetilde{\op{St}}^k_{\tilde{a}_{k-1}}/\tilde\g_k$
for $k>0$ and moreover the source spaces vanish implying
$\widetilde{\op{St}}^k_{\tilde{a}_{k-1}}=\tilde\g_k$ for $k>1$. Now we can use the commutative diagram from
the proof of Proposition \ref{acs2} again. It implies that $\op{St}^{k+1}_{a_k}$ contains both $\g_{k+1}$ and
$\tilde\g_k$, and the stabilization means that nothing more contributes.

We conclude that the parts $\g_k$, $\tilde\g_{k-1}$ of the stabilizer $\op{St}^k_{a_{k-1}}$ resolve upon
prolongation to $k$-jets $\E^k$, but in the new stabilizer $\op{St}^{k+1}_{a_k}$ the parts $\g_{k+1}$, $\tilde\g_k$
appear instead. This proves the claim.
 \end{proof}

For a generic $a_\infty=\{a_k\}\in\E$ the sequence $\{\dim\op{St}^{k+1}_{a_k}\}_{k=0}^\infty$ is equal to
$\{8,16,18,22,26,30,\dots\}$ by a straightforward (albeit very demanding) Maple computation.
Its grows stabilizes starting from $k=3$ in accordance with Proposition \ref{acs3}.
This implies the following dimension formulae for $n=2$ and $k>2$:
 \begin{multline*}
h_k=\dim g_k-\dim\Delta_{k+1}+\dim\op{St}^{k+1}_{a_k}-\dim\op{St}^k_{a_{k-1}}\\
=\dim g_k-\dim\Delta_{k+1}+\dim\g_{k+1}-\dim\g_k+\dim\tilde\g_{k}-\dim\tilde\g_{k-1}\\
=8\binom{k+3}k-4\binom{k+4}{k+1}+4=\frac23k^3+2k^2-\frac83k-4.
 \end{multline*}
For $k\leq2$ we have: $h_0=h_1=0$, but $h_2=2$ (as a straightforward Maple computation verifies).
 %Let us explain how the two scalar invariants of order 2 arise...
This finishes proof of the theorem.
 \end{proof}

Let us list the numbers of the pure order $k$ invariants for the first $n$:
 \begin{center}
\begin{tabular}{c||c|c|c|c|c|c|c|c}
 & $h_0$ & $h_1$ & $h_2$ & $h_3$ & $h_4$ & $h_5$ & $h_6$ & \dots\\
\hline
$n=2$ & 0 & 0 & 2 &    24 &   60 &   116 &   196 & \dots \\
\hline
$n=3$ & 0 & 2 & 64 &   282 &  792 &  1806 &  3612 & \dots \\
\hline
$n=4$ & 0 & 16 & 272 & 1320 & 4392 & 11840 & 27744 & \dots
\end{tabular}
 \end{center}

\medskip

The formulae of the theorem are encoded via the Poincar\'e function:
 $$
P(z)=\left\{
\begin{array}{ll}
\frac{2z^2(1+8z-12z^2+6z^3-z^4)}{(1-z)^4}, & \text{ for }n=2,\vphantom{\frac{\frac22}{\frac22}}\\
\frac{2z(1+26z-36z^2+10z^3+17z^4-18z^5+7z^6-z^7)}{(1-z)^6}, & \text{ for }n=3,\vphantom{\frac{\frac22}{\frac22}}\\
\frac{2n(nz-1)}{z(1-z)^{2n}}+\frac{2n}{z(1-z)^{n-1}}
 % +\frac{n(n-3)(n^2-9n+2)}{12}z
+2n, & \text{ for }n>3
\vphantom{\frac{\frac22}{\frac22}}.
\end{array}
\right.
 $$

%%%%%%%%
\section{Conclusion: Towards the general Arnold conjecture}\label{S4}

By the Hilbert-Serre theorem, the Poincar\'e series of a finitely generated graded module
over an algebra with homogeneous generators of degrees $d_1,\dots,d_n$ has the form
(see \cite{Sp}, also for many examples)
 $$
P(z)=\frac{F(z)}{\prod_{i=1}^n(1-z^{d_i})}.
 $$
This is more general than the one given by formula \eqref{PR}. Indeed, the poles are on the unit circle
$S^1\subset\C$, but can be other roots of unity.

More general Poincar\'e functions arise in the problems of analysis when the pseudogroup $G$ acts
non-transitively on the base. This is the case in singularity theory. For instance,
the pseudogroup of symplectomorphisms $G=\op{Diff}_\text{loc}(\R^{2n},\omega)$ acting on the space of
germs of critical linearly stable Hamiltonians ($\op{Sp}\equiv$ spectrum)
 $$
\E=\{H\in C^\infty_\text{loc}(\R^{2n},0):H(0)=0,d_0H=0,\op{Sp}(\omega^{-1}d^2_0H)\subset i\R\}
 $$
 % point $0\in\R^{2n}$ of the Hamiltonian fields $X_H$ with pure imaginary spectrum of linearization
was considered in \cite{KL}: the corresponding Poincar\'e function on the general stratum
has multiple poles at $\pm1$:
 $$
P(z)=\frac1{(1-z^2)^{n}}.
 $$

Another classical problem is related to the Poincar\'e-Dulac normal form for a vector field $v$ near
stationary point $0\in\R^n(x)$, $v(0)=0$. Let $d_0v$ % be semi-simple with
have spectrum $\Lambda=(\lambda_1,\dots,\lambda_n)$.
Then $v$ is formally equivalent to a vector field $w$ with components (no summation by $i$)
 $$
w^i=\lambda^i x_i+\sum_{m\in R_i(\Lambda)}c^i_m x^m,
 $$
where $R_i(\Lambda)=\{m\in(\mathbb{Z}_{\ge0})^n:m_i=\langle m,\lambda\rangle,|m|=\sum m_i\ge2\}$
is the $i$-th resonance set and
$x^m=x_1^{m_1}\cdots x_n^{m_n}$, see \cite{A2}. Some of the coefficients can be further normalized
leaving only the essential ones.

Clearly, $\op{Sp}(d_0v)=\op{Sp}(d_0w)=\Lambda$ and the normalized coefficients $c^i_m$ are differential invariants.
The corresponding counting function $P(z)$ is rational in all known cases, but it is not arbitrary.
To see this consider the case $n=2$. Here are the main singularities:
 \begin{enumerate}
 \item
Non-resonant case: formal linearization, whence
 $$
P(z)=2z;
 $$
 \item
$\lambda_1/\lambda_2=m$ (or $\frac1m$) with $m\in\mathbb{N}$ fixed, $m>1$ (Poincar\'e domain), then
there is only one non-resonant term \cite{A2} and
 $$
P(z)=z+z^m;
 $$
 \item
$\lambda_1/\lambda_2\in\mathbb{Q}_-$ fixed (Siegel domain: here $\lambda_1,\lambda_2\in\R\cup i\R$),
there is infinity of non-resonant terms, but the normal form leaves only few of them.
A general saddle resonant singularity has the normal form ($p,q\in\N$ are fixed)
$v=p\lambda(x+x^{q+1}y^p)\p_x -q\lambda (y+ax^qy^{p+1}+bx^{2q}y^{2p+1})\p_y$ \cite{VG}
(an elliptic singularity has a similar normal form).
Thus $h_1=h_{m+1}=h_{2m+1}=1$ for $m=p+q$ and $h_i=0$ else, implying
 $$
P(z)=z+z^{m+1}+z^{2m+1};
 $$
 \item
$\lambda_1\neq0,\lambda_2=0$ (or otherwise around) -- the saddle-node point; the normal form here is
$v=\lambda(xQ_m(x)\p_x\pm(y^{m+1}+ay^{2m+1})\p_y$), where $Q_m(x)$ is a polynomial of $\op{deg}=m$ \cite{IY}
and consequently
 $$
P(z)=\frac{z-z^{m+1}}{1-z}+z^{2m+1};
 $$
 \item
$\lambda_1=\lambda_2=0$, but $d_0v\neq0$ (nilpotent linear part), this is the Takens-Bogdanov singularity;
the pre-normal Lienard form is $v=y\p_x+x(xa(x)+yb(x))\p_y$. It can be checked that for $a(0)\neq0\neq b(0)$
a formal change of variables yields $b(x)=0$, $a(x)=\sum_{k\in\Delta}a_kx^k$, where
$\Delta=\{n\in\Z_{\ge0}:n\not\in 3\N+1\}$; alternatively one can eliminate $a(x)$ except for two first terms
and one thrid of the terms of $b(x)$ \cite{WCW}. Both normal forms imply that $h_k$ is the characteristic function $\chi_{\Delta}(k)$, whence
 $$
P(z)=\frac{(1+z+z^2-z^4)z^2}{1-z^3}.
 $$
 \end{enumerate}
Further normal forms for more complicated degenerations can be found in \cite{SZ},
they lead to other rational Poincar\'e functions.

The mechanism explaining this rationality is not the same as in the Lie-Tresse theorem,
see the discussion in \cite{KL}. In Section \ref{S1} we derived the strong form of
Arnold's conjecture provided the pseudogroup $G$ acts transitively on $M$, and we showed many
explicit computations in Section \ref{S3}. It seems plausible that this approach
can be extended to the case when $G$-orbits foliate $M$.
However, in the presence of singular orbits, the general Arnold conjecture is still wide open.

%%%%%%%%%%%%%%%%%%%%%%%%%%%%%%%%%%%%%%%%%%%%%%%%%%%%%%%%%%%%%%%%%%%%%%%%%%

\end{document}